\documentclass[10pt]{amsart}  
\usepackage{amsmath, latexsym, amsthm, amsfonts,bm,amssymb,mathscinet} 
\usepackage{mathtools}
\usepackage{color}
\usepackage[plainpages=false,pdfpagelabels]{hyperref}
 
\usepackage{natbib} 
\usepackage{url} 
\usepackage{enumerate}
\usepackage{units} 
\usepackage{booktabs}

\bibpunct{(}{)}{;}{a}{}{,} 

\theoremstyle{plain}
\newtheorem{thm}{Theorem}
\newtheorem{prop}{Proposition}
\newtheorem{lemma}{Lemma}
\newtheorem{cor}{Corollary}

\theoremstyle{remark}
\newtheorem{rem}{Remark}

\newtheorem{defin}{Definition}
\newtheorem{assump}{Assumption}

\def\convp{\xrightarrow{\mathbb{P}_{b_0,s_0}}}

\def\exz{{ {\mathbb{E}_{0,s_0}\,}}}
\def\exzb{{ {\mathbb{E}_{b_0,s_0}\,}}}

\def\ppz{{ {\mathbb{P}_{0,s_0}\,}}}

\def\varz{{ {\mathbb{V}\mathrm{ar}_{0,s_0}\,}}}
\def\expi{{ {\mathbb{E}_{\Pi_n}\,}}}
\def\vpi{{ {\mathbb{V}\mathrm{ar}_{\Pi_n}\,}}}
\newcommand{\ind}{\mathbf{1}}

\def\pp{{ {\mathbb{P}_{b_0,s_0}\,}}}
\def\ppz{{ {\mathbb{P}_{0,s_0}\,}}}
\def\dd{{ {\mathrm{d}}}}
\def\ee{{ {\mathbb{E}}}}
\def\eps{\varepsilon}

\newcommand{\rr}{\mathbb{R}}
\renewcommand{\th}{\theta}
\newcommand{\si}{\sigma}
\newcommand{\la}{\lambda}

\usepackage[OT2,OT1]{fontenc}
\newcommand\cyr{%
\renewcommand\rmdefault{wncyr}%
\renewcommand\sfdefault{wncyss}%
\renewcommand\encodingdefault{OT2}%
\normalfont
\selectfont}
\DeclareTextFontCommand{\textcyr}{\cyr}

\definecolor{violet}{rgb}{0.3,0.0, 0.55}

\def\epsilon{\varepsilon}

\allowdisplaybreaks

\begin{document}

\title[Bayesian diffusion coefficient estimation]{Nonparametric Bayesian estimation of a H\"older continuous diffusion coefficient}

\author{Shota Gugushvili}
\address{Biometris\\
	Wageningen University \& Research\\
	Postbus 16\\
	6700 AA Wageningen\\
	The Netherlands}
\email{shota@yesdatasolutions.com}

\author{Frank van der Meulen}
\address{Delft Institute of Applied Mathematics\\
Faculty of Electrical Engineering, Mathematics and Computer Science\\
Delft University of Technology\\
Mekelweg 4\\
2628 CD Delft\\
The Netherlands}
\email{f.h.vandermeulen@tudelft.nl}

\author{Moritz Schauer}
\address{Mathematical Institute\\
Leiden University\\
P.O. Box 9512\\
2300 RA Leiden\\
The Netherlands}
\email{m.r.schauer@math.leidenuniv.nl}

\author{Peter Spreij}
\address{Korteweg-de Vries Institute for Mathematics\\
University of Amsterdam\\
P.O. Box 94248\\
1090 GE Amsterdam\\
The Netherlands \and Institute for Mathematics, Astrophysics and Particle Physics\\ Radboud University\\ Nijmegen\\ The Netherlands}
\email{spreij@uva.nl}

\subjclass[2000]{Primary: 62G20, Secondary: 62M05}

\keywords{Diffusion coefficient; Gaussian likelihood; Non-parametric Bayesian estimation; Pseudo-likelihood; Posterior contraction rate; Stochastic differential equation; Volatility}

\begin{abstract}
We consider a nonparametric Bayesian approach to estimate the diffusion coefficient of a stochastic differential equation given discrete time observations  over a fixed time interval.
As a prior on the diffusion coefficient, we employ a histogram-type prior with piecewise constant realisations on bins forming a partition of the time interval. Specifically, these constants are realizations of independent inverse Gamma distributed randoma variables.
We justify our approach by deriving the rate at which  the corresponding posterior distribution asymptotically concentrates around the data-generating diffusion coefficient. This posterior contraction rate turns out to be optimal for estimation of a H\"older-continuous diffusion coefficient with smoothness parameter $0<\lambda\leq 1.$ Our approach is straightforward to implement, as the posterior distributions turn out to be inverse Gamma again, and leads to good practical results in a wide range of simulation examples. Finally, we apply our method on exchange rate data sets.

\end{abstract}

\date{\today}

\maketitle
\tableofcontents

\section{Introduction}

\subsection{Problem description} Stochastic differential equations (SDEs) have been widely used as models in numerous applications ranging from physics (see for example \cite{allen}) to engineering (see \cite{wong}) and to finance (see \cite{rutkowski}). We assume  observations from an SDE of the form
\begin{equation}
\label{sde}
\dd X_t=b_0(t,X_t)\,\dd t+s_0(t)\,\dd W_t, \quad X_0=x, \quad t\in[0,T],
\end{equation}
with a drift coefficient $b_0,$ (deterministic) dispersion coefficient $s_0,$ and a deterministic initial condition $x.$ Here $X$ is real valued and $W$ is a Brownian motion. 
We assume observations \[\scr{X}_n =\{X_{t_{0,n}}, \ldots, X_{t_{n,n}}\}\] from the solution $X$ to \eqref{sde} are available, where $t_{i,n}=\tfrac{i}n T,$ $i=0,\ldots,n.$ Our aim is to estimate $s_0$ nonparametrically within the Bayesian setup. 

Model \eqref{sde} covers the case of linear SDEs, such as the popular Ornstein-Uhlenbeck process; see, e.g., Section 5.6 in \cite{karatzas}. Among references that study \eqref{sde} as models for log returns of asset prices, we mention \cite{lutz10} and \cite{mishura15}, but the model has applications far beyond this context as well. A seemingly more general SDE 
\begin{equation}
\label{sde2}
\dd X_t=\widetilde{b}_0(t,X_t)\,\dd t+s_0(t)f_0(X_t)\,\dd W_t, \quad X_0=x, \quad t\in[0,T],
\end{equation}
can be reduced to the form \eqref{sde} through a simple transformation of $X_t$, namely
\begin{equation*}
Y_t=x+\int_0^{X_t}\frac{1}{f_0(u)}\dd u,
\end{equation*}
provided $f_0$ is known and sufficiently regular; see p.~186 in \cite{soulier98}. Financial practitioners often are content with a model of the type \eqref{sde2} as a simple and useful generalisation of the Black-Scholes model; see, e.g., pp.~7--14 in \cite{gatheral} for Dupire, Derman and Kani's pioneering work on local volatility. In particular, a discretely observed geometric Brownian motion with time-varying coefficients  is also a special case of \eqref{sde2}, once one passes to the corresponding log returns (see \cite{taleb} for additional information and applications). As the drift in \eqref{sde} is allowed to be non-linear, the distribution of $X_t$ is in general not Gaussian and may well exhibit heavy tails, which is attractive from the point of view of financial applications. Finally, in some practical applications it is genuinely important to employ a time-dependent diffusion coefficient; a real data example is given in  Section \ref{section:realdata}.

\subsection{Related literature} Statistical inference for SDEs is a well studied and very active field of research, that is far from saturation. Relevant literature can be divided into two categories: works dealing with parametric and works dealing with nonparametric methods. Parametric approaches specify parametric forms for the drift and diffusion coefficients of SDEs. When these specifications use correct functional forms, such methods attain a higher statistical efficiency over the nonparametric ones. On the other hand, nonparametric approaches, where one only assumes qualitative features of the drift and diffusion coefficients, guard one against model misspecification, which may have dramatically negative consequences for valid inference; also, nonparametric techniques may suggest plausible parametric models in those situations where these models cannot be derived from the first principles; see, e.g., \cite{silverman86}. For parametric approaches to inference in SDE models, see, e.g., Chapter 2 in \cite{kutoyants}, Chapter 3 in \cite{iacus}, and references therein. Nonparametric statistical inference for SDEs of the type studied in the present work has been considered in \cite{genon92}, \cite{hoffmann97} and \cite{soulier98} within the frequentist setup, while \cite{gugu} and \cite{gugu16} have explored the problem from the Bayesian perspective. Although the nonparametric methods these papers study are implementable in principle, these works are primarily of theoretical nature and practical performance of the corresponding approaches is not clear. Furthermore, except \cite{gugu} and \cite{gugu16}, there is hardly any other work available on estimation of the dispersion coefficient (or diffusion coefficient) from the nonparametric Bayesian point of view, which constitutes the central topic of our paper. In this context we can mention only a theoretical contribution \cite{nickl} and a practically oriented paper \cite{batz}, but the models considered there, as well as the sampling scheme, are different from ours, and the theory developed in \cite{nickl} does not cover the approach in \cite{batz}. On a general level, apart from a philosophical appeal for Bayesians, advantages of a Bayesian approach include automatic quantification of uncertainty in parameter estimates through Bayesian credible sets, and the fact that it is a fundamentally likelihood-based method (see \cite{berger}). Furthermore, recent practical advances made in nonparametric Bayesian estimation of the drift coefficient, see, e.g., \cite{frank14},  \cite{frank17a} and \cite{papa}, 
would suggest that comparable results can be obtained for estimation of the diffusion coefficient too. We note, however, that from an implementational point of view nonparametric Bayesian estimation of a dispersion coefficient is very different from drift coefficient estimation: the latter fundamentally relies on the equivalence of laws of continuously observed diffusion processes that have the same diffusion coefficient, which is not applicable when the diffusion coefficient itself is unknown and is a parameter to be estimated.

\subsection{Approach and results} The main practical challenges for Bayesian inference in SDE models from discrete observations are an \emph{intractable likelihood} and absence of a closed form expression for the posterior distribution, which complicates considerably the inference; see, e.g., \cite{roberts-stramer}, \cite{elerian01}, \cite{fuchs} and \cite{frank17a}. We circumvent these difficulties by intentionally misspecifying the drift coefficient, and employing a (conjugate) histogram-type prior on the diffusion coefficient, that has piecewise constant realisations on bins forming a partition of $[0,T]$ (this is different from \cite{gugu} and \cite{gugu16}, where the drift $b_0$ is in fact zero, and other priors are used). Due to this, our nonparametric Bayesian method to estimate the dispersion coefficient $s_0$ in \eqref{sde}  is  easily  implemented, fast and requires little fine-tuning from the user. We demonstrate its good practical performance on a wide range of simulated data examples and we apply it on real data from finance, yielding interesting conclusions.

On the theoretical side, we investigate the asymptotic performance of our Bayesian procedure from the frequentist point of view. Theoretical analysis of Bayesian procedures for inference in SDE models from discrete observations is in general challenging; see, e.g., the contributions \cite{frank13}, \cite{gugu14} and \cite{nickl} for an impression, albeit in settings different from ours. We consider the `infill' asymptotics with the time $T$ horizon staying fixed and the number of observations $n$ in the interval $[0,T]$ increasing; this asymptotic regime is standard in the literature and can be thought of as reasonably satisfied in many financial applications. Complicating factors for a theoretical analysis in our setting are due to influence of the unknown drift coefficient $b_0$, which we have intentionally misspecified. We address this through an argument based on Girsanov's theorem. The main theoretical result we obtain tells us that the drift misspecification is asymptotically harmless and our procedure for estimating the diffusion coefficient is consistent at rate $n^{-\beta}$ in the $L_2$-norm, with the precise value of $\beta$ depending on the smoothness of the true dispersion coefficient. The corresponding posterior contraction rate is optimal for estimation of H\"older smooth dispersion coefficients of order $0<\lambda\leq 1$.

\subsection{Organisation of this paper}  Section~\ref{section:model} contains the model specification and a detailed description of our nonparametric Bayesian approach. Section~\ref{section:asymp} contains the theoretical results. Our Bayesian method depends on a hyperparameter, the number of bins forming the partition of the interval $[0,T]$, and Sections~\ref{sec:dic} and \ref{sec:bayes:factor} discuss possible practical methods of its choice. In Section~\ref{section:simulation} we investigate practical performance of our method via simulations and provide illustrations of our theory from Section~\ref{section:asymp}. Examples with real data are studied in Section~\ref{section:realdata}. Proofs of the results from Section~\ref{section:asymp} can be found in Section~\ref{section:proofsgeneral}. In Appendix~\ref{appendix:asymp} we state and prove an additional theoretical result.

\subsection{Frequently used notation}
We denote by $\|\cdot\|_2$  the $L_2$-norm with respect to the Lebesgue measure on the Borel sets of $[0,1]$.
We use the following notation to compare two sequences $\{a_n\}$ and $\{b_n\}$ of positive real numbers: $a_n\lesssim b_n$  (or $b_n\gtrsim a_n$) means that there exists a constant $C>0$ that is independent of $n$ and is such that $a_n\leq C b_n.$ As a combination of the two we write $a_n\asymp b_n$ if both $a_n\lesssim b_n$ and $a_n\gtrsim b_n$. We will also write $a_n \gg b_n$ to indicate that $a_n/b_n\rightarrow\infty$ as $n\rightarrow\infty$. By $a \vee b$ we denote the maximum of two numbers $a$ and $b.$ 

We let  $\pp$ denote the law of the path $(X_t\colon t\in[0,T])$ from \eqref{sde} under the true parameter values $(b_0,s_0).$ In particular the notation $\mathbb{P}_{0,s_0}$ is used for such a law when the drift coefficient is equal to zero.
We denote the prior distribution on the dispersion coefficient by $\Pi_n$ (with $n$ the number of observations) and write the posterior as $\Pi_n(\,\cdot\mid \mathcal{X}_n)$.
 We denote the posterior expectation and variance by $\expi(\,\cdot\mid\mathcal{X}_n)$ and $\vpi(\,\cdot\mid \mathcal{X}_n),$ respectively.

\section{Assumptions and Bayesian setup}\label{section:model}

We summarise the assumptions on our statistical model.

\begin{assump}
\label{standing}
Assume that
\begin{enumerate}[(a)]
\item the model \eqref{sde} is given with $x=0$ and $T=1$;
\item the drift coefficient satisfies a linear growth condition and is Lipschitz in its second argument: for some $K>0$ it holds that
\begin{align*}
|b_0(t,x)|^2 & \leq K(1+|x|^2),\quad \forall x\in\rr,\\
|b_0(t,x)-b_0(t,y)| &\leq K|x-y|, \quad \forall t\in[0,1], \quad \forall x,y\in\rr;
\end{align*}
\item the dispersion coefficient $s_0$ is H\"older continuous on $[0,1]$ with H\"older constant $L$ and H\"older exponent $\lambda\in (0,1]$, $|s_0(u)-s_0(v)|\leq L|u-v|^\lambda$ for all $u,v\in [0,1]$,
and is bounded away from zero and (by continuity also from infinity);
\item a discrete time sample
\[ \scr{X}_n =\{X_{t_{0,n}}, \ldots, X_{t_{n,n}}\}\] from the solution $X$ to \eqref{sde} is available, where $t_{i,n}=i/n,$ $i=0,\ldots,n.$ For future reference, we also define $Y_{i,n} = X_{t_{i,n}}-X_{t_{i-1,n}}$.
\end{enumerate}
\end{assump}

Under Assumption~\ref{standing}, equation \eqref{sde} admits a unique strong solution, see, e.g., Theorems 2.5 and 2.9 of section 5.2 in \cite{karatzas}. The minimal regularity conditions on the dispersion coefficient $s_0$ in Assumption \ref{standing}~(c) are needed for our asymptotic statistical theory to work; see Section~\ref{section:asymp}. The sampling scheme in Assumption \ref{standing}~(d) is referred to as the high-frequency data setting and is a popular asymptotic setup for inference in SDE models, see, e.g., \cite{dette06}, \cite{florens93}, \cite{genon92}, \cite{hoffmann97}, \cite{hoffmann99}, \cite{jacod00} and \cite{soulier98}. Financial data are often of a similar type, see \cite{sabel}. As remarked in \cite{ignatieva12}, p.~1334, ``while the estimation of a drift coefficient function is theoretically of major interest in finance, for instance, for portfolio optimization, it can practically rarely be achieved. It is a fact that accurate estimation of drift coefficient functions requires considerably longer time series than typically available". This example, then, provides an instance of a setting when the assumption $T\rightarrow\infty$ cannot be thought to be reasonably satisfied. Other sampling schemes have been also considered in the literature on nonparametric inference for SDE models, e.g.\ when the time horizon $T\rightarrow\infty$, but the distance $\Delta$ between observation times stays fixed (the so-called low frequency data setting; see, e.g., \cite{gobet04} and \cite{nickl}). Alternatively, one can assume observations are available at times $\Delta,2\Delta,\ldots$, and consider the case $\Delta\rightarrow 0$, $T=n\Delta\rightarrow\infty$, see, e.g., \cite{hoffmann} (this is again referred to as the high-frequency data setting). The question which of the asymptotic regimes reasonably applies to a given dataset in practice can be decided only on the case by case basis.

Our first observation is that under the sampling scheme as in Assumption \ref{standing}~(d), consistent estimation of the drift coefficient $b_0$ is impossible; see, e.g., \cite{mai}, p.~919. Furthermore, in many contexts, e.g.\ when pricing financial derivatives, knowledge of the drift coefficient is in fact of no interest, whereas the dispersion coefficient is of paramount importance; see \cite{rutkowski}. This motivates us to completely ignore the drift coefficient in our estimation procedure by intentionally misspecifying the model and acting as if the drift were equal to zero. We will justify this in Section~\ref{section:asymp}. Then the pseudo-likelihood associated with our observations is Gaussian and is given by
\begin{equation}
\label{likelih}
L_n(s)=\prod_{i=1}^{n} \left\{ \frac{1}{\sqrt{2\pi\int_{t_{i-1,n}}^{t_{i,n}}s^2(u)\,\dd u}}\psi\left( \frac{Y_{i,n}}{\sqrt{\int_{t_{i-1,n}}^{t_{i,n}}s^2(u)\,\dd u}} \right) \right\},
\end{equation}
where $\psi(u)=\exp(-u^2/2).$ Gaussian pseudo-likelihood is a widely used object in statistics, see, e.g., Section 5.2 in \cite{brockwell}, \cite{dimitriou} and \cite{hualde} for some examples. 
Setting the drift to zero in our setting has a practical advantage of obtaining a simple and tractable expression for the pseudo-likelihood. 

With $\Pi_n$ denoting a prior on the dispersion coefficient, provided all the involved quantities are suitably measurable, Bayes' theorem gives that the posterior probability of any measurable subset $S\subset\mathcal{S}$ of dispersion coefficients is given by
\begin{equation*}
\Pi_n(S\mid X_{t_{0,n}}\ldots,X_{t_{n,n}})=\frac{\int_{S} L_n(s) \Pi_n(\dd s)}{ \int_{\mathcal{S}} L_n(s) \Pi_n(\dd s) },
\end{equation*}
where $\mathcal{S}$ denotes a space on which the prior $\Pi_n$ is defined. From the above display various point estimates of $s_0$ can be obtained, such as the posterior mean.

It follows from \eqref{likelih} that the likelihood depends on the parameter of interest only through the integrals $\int_{t_{i-1,n}}^{t_{i,n}}s^2(u)\,\dd u$ and is otherwise `blind' to precise values the diffusion coefficient takes inside the intervals $[{t_{i-1,n}},{t_{i,n}}]$. Consequently, it appears natural to a priori model the diffusion coefficient as piecewise constant on intervals $[{t_{i-1,n}},{t_{i,n}}]$. However, some smoothing should also be performed, and this can be achieved by aggregated several neighbouring intervals $[{t_{i-1,n}},{t_{i,n}}]$. Thus, to construct the prior, we proceed as follows: Let $m$ be an integer smaller than $n$. 
Then we can uniquely write $n=mN+r$  with $0\leq r<m$, and in fact $N=\lfloor \frac{n}{m}\rfloor$. 
 Both $m$ and $N$ will depend on $n$ (and we also write $m_n$ and $N_n$ to emphasize this when appropriate). With this assumption we have bins $B_k=[t_{m(k-1),n},t_{mk,n})$, $k=1,\ldots,N-1$ and $B_N=[t_{m(N-1),n},1]$. Note   that  the length of $B_k$ is then equal to $m/n$ for $k\leq N-1$, whereas $B_N$ has length $1-t_{m(N-1),n}=\frac{r+m}{n}<\frac{2m}{n}$.
For notational convenience later on, we also write
\begin{align*}
B_k&=[a_{k-1},a_k), \quad k=1,\ldots,N-1,\\ 
B_{N}&=[a_{N-1},a_{N}]. 
\end{align*}
Let
$
s=\sum_{k=1}^{N_n} \xi_k \ind_{B_k}.
$
The prior $\Pi_n$ on the dispersion coefficient $s$ is defined by putting a prior on the coefficients $\xi_k$'s. Since
\begin{equation}
\label{eq:series}
s^2=\sum_{k=1}^{N_n} \xi_k^2 \ind_{B_k}=\sum_{k=1}^{N_n} \theta_k \ind_{B_k},
\end{equation}
where we have put $\theta_k=\xi_k^2$, equivalently one can place the prior on the coefficients $\theta_k$'s of the diffusion coefficient $s^2$.

We call the prior $\Pi_n$ a histogram-type prior. Conceptually somewhat similar priors have already been employed in the nonparametric Bayesian density estimation context, as well as the Poisson intensity estimation context, see, e.g., \cite{scricciolo03}, \cite{scricciolo04}, \cite{scricciolo07}, \cite{arjas97}, \cite{heikkinen98}, \cite{castillo14}, \cite{castillo15} and \cite{gine11}, but our problem is rather different from density or Poisson intensity estimation and requires the use of many different ideas. In practice, e.g.\ in financial applications, one can hardly hope to estimate the diffusion coefficient to a very fine degree of detail, and in that respect using a prior with piecewise constant realisations does not appear to be unnatural. We note that a practical frequentist approach to nonparametric volatility estimation in \cite{sabel} likewise produces piecewise constant estimates. There is yet another practical advantage in using priors with piecewise constant realisations: financial time series often exhibit jumps, accurate detection of which is a delicate task. Due to their localised nature, our Bayesian estimates of the dispersion coefficient are likely to quickly recover from negative effects of a moderate number of undetected jumps.

If the prior coefficients $\th_1,\ldots, \th_N$ are independent and have an inverse gamma $\mathrm{IG}(\alpha,\beta)$ distribution with parameters $\alpha,\beta>0$, which will henceforth be our assumption, then the posterior is conjugate, as stated in the next lemma.
\begin{lemma}\label{lem:post-ig}
Assume $\th_1,\ldots, \th_N$ are independent with the inverse gamma $\mathrm{IG}(\alpha,\beta)$ distribution. Then $\th_1,\ldots, \th_N$ are a posteriori independent and, for $k=1,\ldots,N-1$,
\[ \th_k \mid \scr{X}_n \sim \mathrm{IG}\left(\alpha+m/2, \beta+ n Z_k/2\right). \]
with 
\[ Z_k = \sum_{i=(k-1)m+1}^{km} Y_{i,n}^2, \]
whereas
\[ \th_N \mid \scr{X}_n \sim \mathrm{IG}\left(\alpha+(m+r)/2, \beta+ n Z_N/2\right), \]
with 
\[ Z_N = \sum_{i=(N-1)m+1}^{n} Y_{i,n}^2. \]
\end{lemma}

\label{section:proofsig}
\begin{proof}
Write $\theta=(\theta_1,\ldots,\theta_N)$.
The  likelihood, considered as a function of  $\theta$, is given by 
\begin{multline*} L_n(\th) =\prod_{i=1}^n \phi\left(0; Y_{i,n}, \int_{t_{i-1}}^{t_i} s^2(u) d u \right)\\ \propto  \theta_N^{-(m+r)/2} \exp\left(-\frac{n Z_N}{2\th_N} \right) \prod_{k=1}^{N-1} \th_k^{-m/2} \exp\left(-\frac{n Z_k}{2\th_k} \right). \end{multline*}
Here $\phi(x; \mu, \si^2)$ denotes the density of the normal distribution with mean $\mu$ and variance $\si^2$ evaluated at point $x$. As the prior satisfies
$p(\th_1,\ldots, \th_N) \propto \prod_{k=1}^N \th_k^{-(\alpha+1)} e^{-\beta/\th_k},$
the result easily follows. 
\end{proof}

Posterior computations thus turn out to be elementary with our approach. For instance, the posterior mean of $s^2$ can be obtained from the posterior means of $\theta_k$'s. Recall that the $\mathrm{IG}(\alpha,\beta)$ distribution has mean $\beta/(\alpha-1)$ and variance $\beta^2/(\alpha-1)^2(\alpha-2)$ (if finite, to which end one must have $\alpha>2$). Then, e.g., for $k<N$ and $m\geq 2$, the posterior mean of $\theta_k$ is equal to
\begin{equation}\label{postgamma}
\ee_{\Pi_n}(\theta_k\mid\mathcal{X}_n)=\frac{\beta+ n Z_k/2}{\alpha+m/2-1}.
\end{equation}
Conceptually the posterior mean of $s^2$ in this context is similar to a regressogram; see, e.g., Examples 4.5 and 5.24 in \cite{wasserman06}. However, the Bayesian approach deals with the \emph{entire posterior distribution} and does not reduce to a point estimate, such as the posterior mean.

\begin{rem}
Instead of a histogram-type representation in \eqref{eq:series}, one could have tried to base the prior on some other series representation for $s^2$. At first sight, e.g.\ splines are a sensible choice to that end. However, enforcing spline basis functions to have disjoint supports with endpoints at $t_{n,i}$'s does not appear to be a natural procedure, whereas other choices would not have lead to simple posterior computations as in Lemma \ref{lem:post-ig}.
\end{rem}

\section{Asymptotic theory}\label{section:asymp}

\subsection{Generalities}\label{section:general}

A highly desirable property of a Bayesian procedure, in particular from the frequentist point of view, is that the posterior asymptotically concentrates around the true parameter value. In fact, studying the rate at which the posterior contracts around the true parameter is similar to studying  convergence rates of frequentist estimators. Some of by now classical references, where general conditions for derivation of posterior contraction rates are given, include \cite{ghosal00}, \cite{ghosal07} and \cite{shen01}. However, we will follow a rather different and more direct path of the proof.

For $\varepsilon>0,$ we denote by
\begin{equation*}
U_{s_0,\varepsilon}=\left\{ s\in\mathcal{S}_n\colon\|s-s_0\|_2<\varepsilon \right\}
\end{equation*}
the  $L_2$-neighbourhood of $s_0$ of radius $\varepsilon$. 

In the next proposition, we show that without loss of generality one may assume $b_0=0$ in the proofs. This proposition also explains why ignoring the drift in our estimation procedure by intentionally setting it to zero still leads to consistent Bayesian estimation of $s_0.$ The corresponding theoretical argument relies on an application of Girsanov's theorem (see, e.g., Section 3.5 in \cite{karatzas}). We would like to stress the fact that given the simplicity of our prior and intentional misspecification of the likelihood, the possibility of consistent estimation of $s_0$ with our approach is not obvious and requires a thorough investigation. 

\begin{prop}
\label{prop1}
Let Assumption~\ref{standing} hold and assume that for $\varepsilon_n\rightarrow 0$
\begin{equation*}
\ee_{0,s_0}[\Pi_n(U_{s_0,\varepsilon_n}^c\mid\mathcal{X}_n) ] \rightarrow  0
\end{equation*}
as $n\rightarrow\infty.$
Then also
\begin{equation*}
\ee_{b,s_0}[\Pi_n(U_{s_0,\varepsilon_n}^c\mid\mathcal{X}_n)] \rightarrow 0.
\end{equation*}
\end{prop}

\subsection{Posterior contraction rates}\label{section:ig}
As we consider the asymptotics $n\rightarrow\infty$, we take the number of bins $N$ to depend on the sample size $n$, and indicate this in our notation by writing $N_n$. Then also $m$ depends on $n$, and we write $m_n$ to emphasise this dependence.

Assume that Assumption \ref{standing} holds for the remainder of this subsection.  The following theorem shows that posterior contracts at rate $n^{-\lambda/(2\lambda+1)}$ in $L_2$.

\begin{thm}\label{thm:rate-posterior-ig}\label{cor:rate:b}
  Assume  $N_n \asymp n^{1/(2\lambda+1)}$.  If we let  $ \eps_n\asymp n^{-\lambda/(2\lambda+1)}$, then for any sequence $h_n$ tending to infinity (as $n\to \infty$) we have
\[ 
\exzb \left[ \Pi_n(\|s^2-s_0^2\|_2 \ge h_n \eps_n \mid \scr{X}_n) \right]  \to 0
\]
as $n\to \infty.$
\end{thm}

In the next theorem we give the posterior contraction rate for the sup-norm.

\begin{thm}\label{cor:rate:b:sup}
Assume  $N_n\asymp n^{1/(2\lambda+1)}$.   
If we let  $ \tilde\eps_n\asymp n^{-\lambda/(2\lambda+2)}$, then for any sequence $h_n$ tending to infinity (as $n\to \infty$) we have 
\[ \exzb \left[ \Pi_n\Big(\sup_{x\in [0,1]}|s^2(x)-s_0^2(x)| \ge h_n \tilde\eps_n \mid \scr{X}_n\Big) \right]  \to 0 \]
as $n\to \infty.$
\end{thm}

\subsection{Discussion}
Now we provide some discussion on the obtained theoretical results.

\begin{rem}
Establishing posterior contraction in the $L_2$-metric is rather natural, as $\int_0^1 s_0^2(t)\dd t$ is the quadratic variation of the process $(X_t:t\in [0,1])$ over the interval $[0,1].$ It is also the variance of $X_1$ when the drift coefficient $b_0$ is a deterministic function depending only on time.
\end{rem}

\begin{rem}
The inequality
\begin{equation*}
\|s^2-s_0^2\|_2 \geq \kappa\|s-s_0\|_2 ,
\end{equation*}
valid for $s_0$ satisfying Assumption \ref{standing} (here $\kappa>0$ is a lower bound of $s_0$), implies that the rate of Theorem~\ref{thm:rate-posterior-ig} is also valid with $\|s-s_0\|_2$. A similar remark applies to the posterior contraction in Theorem~\ref{cor:rate:b:sup}.
\end{rem}

\begin{rem}
A comparison with the frequentist minimax convergence rate in \cite{hoffmann97} shows that the posterior for the diffusion coefficient contracts at the optimal rate in the $L_2$-metric (strictly speaking, the results in the latter paper are given for $B^{s}_{p,q}$-Besov smooth diffusion coefficients with $s>1$, but general arguments for derivation of lower bounds in our statistical setup are classical and work also in the H\"older setting). With histogram-type priors considered in this work no further improvement in the posterior contraction rate, when $\lambda=1$, is possible beyond $n^{-1/3},$ even if the function $s_0$ is smoother than a Lipschitz function. An intuitive reason for this is that realisations of our histogram-type priors are too rough for this; cf.\ p.~629 in \cite{scricciolo07}.
\end{rem}

\begin{rem}
There exists an excellent and deep reference on Bayesian inference in misspecified infinite-dimensional models, namely \cite{kleijn}. That paper provides some additional intuition why our approach is still successful despite the model misspecification. Summarised somewhat simplistically, the results from \cite{kleijn} say that the posterior in misspecified statistical models asymptotically concentrates around that value from the parameter space that is closest to the `true' parameter value in the sense of the minimal Kullback-Leibler distance between respective probability distributions. Asymptotically the observation scheme as in Assumption \ref{standing}~(d) is almost as good as observing the process $X$ continuously over the interval $[0,1].$ On the other hand, the laws corresponding to paths $(X_t:t \in[0,1])$ with two different diffusion coefficients are mutually singular, see Theorem 3.24 in  \S III.3d, \cite{jacod}, with a consequence that the corresponding Kullback-Leibler divergence is infinite. Hence, irrespective of the prior assumptions on the drift coefficient, the posterior for the dispersion coefficient (equivalently, diffusion coefficient) should concentrate around the `true' dispersion coefficient $s_0$ (equivalently, the true diffusion coefficient $s_0^2$), for it is precisely this parameter value that yields  finite Kullback-Leibler divergence between the laws of the `true' and various misspecified models.
\end{rem}

\begin{rem}
The result in Theorem~\ref{cor:rate:b:sup} has to be compared to similar results on estimation of the volatility (not necessarily a deterministic function of time, as in our paper) in the $L_{\infty}$-norm. Such results  have been obtained in \cite{ait2014}, \cite{kanaya16}, \cite{kristensen10} and \cite{malliavin09}. As observed in \cite{ait2014}, p.~273, `near to nothing is known on this topic'. \cite{malliavin09} establish consistency of their Fourier-based estimator of volatility, without deriving its convergence rate. The asymptotics considered in \cite{kanaya16} are different from the one in the present work, which makes a direct comparison difficult. Finally, Theorem 3.5 in \cite{kristensen10} gives a convergence rate of a  kernel estimator of volatility; when translated to our setting, the corresponding optimal convergence rate, up to a log factor, is $n^{-\lambda/(2\lambda + 1)}$. This is a faster rate than the rate obtained in our Theorem~\ref{cor:rate:b:sup}. However, the suboptimal rate in Theorem~\ref{cor:rate:b:sup} is likely to be an artefact of our proof, specifically a somewhat rough bound employed in the second inequality of \eqref{rate:loss}. Unlike our method, the kernel estimator in \cite{kristensen10} suffers from the boundary bias problem, which necessitates studying its asymptotic properties on a time interval strictly contained in $[0,T]$.
\end{rem}


\begin{rem}
In \cite{soulier98}, the following frequentist estimator of $s_0^2$ is introduced,
\begin{equation*}
\hat{s}^2(t)=\sum_{i=1}^{n} K_h(t,t_i) Y_{i,n}^2, 
\end{equation*}
where $K$ is a kernel function, a constant $h>0$ is a bandwidth, and
$K_h(s,t)=\frac{1}{h}K\left(\frac{t-s}{h}\right)$
is a rescaled kernel. Suppose now $K$ is a boxcar kernel, $K(u)=(1/2) 1_{|u|\leq 1},$ see p.~55 in \cite{wasserman06}. Then
\begin{equation*}
\hat{s}^2(t)=\frac{1}{2h} \sum_{i\colon|t-t_{i,n}|\leq h} Y_{i,n}^2.
\end{equation*}
A brief reflection shows that for $k<N$, $n$ large, and $h=m/(2n)$ (half the bin length),  $\hat{s}^2$ is quite similar to \eqref{postgamma}, the difference being that in that formula averaging occurs over individual bins, while here one averages locally over observations in a neighbourhood of each time point $t.$ We note, however, that the asymptotic theory in \cite{soulier98} does not cover the asymptotics of the posterior mean as in \eqref{postgamma}, and also that the regularity conditions of that paper are different from ours. On the other hand, practical computation of the kernel estimator $\hat{s}^2$ would typically require from a user some form of data binning, cf.\ Appendix D.2 in \cite{wand95}, so that from this point of view the posterior mean and the estimator $\hat{s}^2$ are closely related.
\end{rem}

\begin{rem}
We have already pointed out in the introduction that our setup covers more general SDE models than those with deterministic dispersion coefficients, see equation \eqref{sde2}; this follows from an application of It\^{o}'s formula. Under regularity conditions, a further generalisation of our results is possible to the case when the dispersion coefficient $s_0$ is in fact a stochastic process independent of the driving Wiener process in \eqref{sde} (we cite an interesting example in this context: a widely known stochastic volatility model arising as a diffusion limit of a $\operatorname{GARCH}(1,1)$ process, see \cite{nelson90}). Namely, in this case one can simply follow the Bayesian methodology we described in Section \ref{section:model} with no further changes, acting as if the dispersion coefficient were a deterministic function. Despite (yet another) purposeful misspecification, the resulting Bayesian procedure is consistent. This can be established by combining arguments in the present paper with the ones similar to those in \cite{kristensen10}, that deal with a kernel volatility estimator. Indeed, careful examination shows that one of the main steps in our proofs is what can be termed a Bayesian bias-variance decomposition, see the proof of Theorem \ref{thm:rate-posterior-ig}. Both the `bias' and `variance' terms there are analysed using techniques similar to those employed in e.g. kernel regression or density estimation problems, whence a possibility for further generalisations. Space considerations preclude us from studying this interesting question in detail in the present work. 
\end{rem}

\section{Bin number selection via DIC}
\label{sec:dic}


In this section we describe a method of choosing $N$ that is based on the Deviance Information Criterion (DIC) of \cite{spiegelhalter02}; see also \cite{spiegelhalter14} and \cite{gelman14}. Further discussion is given in Section \ref{section:simulation}.

By $M$ we denote the posterior mean of $s^2$; $\log L_n(M)$ is our notation for the log-likelihood evaluated at the posterior mean. Introduce the DIC measure of predictive accuracy,
\[
\widehat{\operatorname{elpd}}_{\textrm{DIC}}=\log L_n(M)-\nu_{\textrm{DIC}},
\]
where ``elpd'' is an abbreviation for ``expected log predictive density'' and 
\[
\nu_{\textrm{DIC}}=2 \left\{ \log L_n(M) - \expi ( \log L_n(s) | \mathcal{X}_n ) \right\}
\]
is the effective number of parameters. Straightforward but tedious calculations employing Lemma \ref{lem:post-ig} and properties of the (inverse) gamma distribution give that
\begin{align*}
\log L_n(M)&=-\frac{n}{2}\log(2\pi)-\frac{n}{2}\log\left(\frac{T}{n}\right)\\
& \quad -\frac{1}{2}\sum_{k=1}^N m_k \log\left( \frac{\beta+nZ_k/2}{\alpha+m_k/2-1} \right)-\frac{n}{2T}  \sum_{k=1}^N Z_k \frac{\alpha+m_k/2-1}{\beta+nZ_k/2},
\end{align*}
where $m_k$ denotes the number of observations in the bin $B_k$ (with a harmless abuse of notation) and $Z_k$ is as in Lemma \ref{lem:post-ig}. By similar calculations,
\[
\nu_{\textrm{DIC}}=\frac{n}{T}\sum_{k=1}^N \frac{Z_k}{\beta+nZ_k/2}-\sum_{k=1}^N m_k \left\{ \Psi\left(\alpha+\frac{m_k}{2}\right) - \log\left( \alpha+\frac{m_k}{2}-1 \right) \right\},
\]
where $\Psi$ is the digamma function. The formulae simplify even further, when $m_k=m,$ $k=1,\ldots,N.$

Now the idea consists in evaluating $\widehat{\operatorname{elpd}}_{\textrm{DIC}}$ for a range of values of $N,$ and choosing the one that maximises $\widehat{\operatorname{elpd}}_{\textrm{DIC}}$. This aims at optimising the predictive behaviour of the model. We note that using predictive performance criteria for Bayesian model selection is a well-established practice in the case of finite-dimensional, parametric models (see, e.g., \cite{gelman14}), but seems to be a new idea in the non-parametric Bayesian setting. DIC in some sense constitutes a Bayesian analogue of Akaike's AIC, and conceptually our proposal is similar to employing information criteria for smoothing parameter selection in the frequentist literature; see, e.g., the widely cited work \cite{hurvich98}. On the computational side, our DIC-based method is very simple to implement and does not require heavy computations. We test its practical performance in Section~\ref{section:simulation}.

\section{Bin number selection via marginal likelihood}
\label{sec:bayes:factor}

In this section we describe an alternative method of the bin number selection to the one we discussed in Section~\ref{sec:dic}. This is based on maximising the marginal likelihood
$
\int_{\mathcal{S}} L_n(s) \Pi_n(\dd s)
$
as a function of $N$, which can be viewed as model evidence given the data. Model selection based on the marginal likelihood (or Bayes factors) is well-established in Bayesian statistics. See, e.g., \cite{wang12} for an application to smoothing parameter selection in the context of smoothing spline regression, which is conceptually related to choosing the bin number $N$ in our problem. On a more general level, this is nothing else but an instance of a well-known empirical Bayes method (cf.~\cite{gelman}, Section~5.1).

Since we identify a piecewise constant diffusion coefficient $s^2$ with its coefficients $\theta_1,\ldots,\theta_N$, using a priori independence of $\theta_k$'s and Fubini's theorem, the marginal likelihood in our setting can be evaluated as (here $m_k$ was defined in Section~\ref{sec:dic} and is the number of observations in bin $B_k$)
\begin{multline*}
\operatorname{ML}_N(\mathcal{X}_n)=\prod_{k=1}^N \left\{ \int_{[0,\infty)} \left(\frac{2\pi}{n}\right)^{-m_k/2} \frac{\beta^{\alpha}}{\Gamma(\alpha)} \theta_k^{-(\alpha+m_k/2)-1} \exp\left(- \frac{1}{\theta_k}\left( \beta+\frac{nZ_k}{2} \right) \right) \dd \theta_k \right\} \\
\propto \frac{\beta^{\alpha N}}{\Gamma(\alpha)^N} \prod_{k=1}^N \frac{\Gamma(\alpha+m_k/2)}{(\beta+nZ_k/2)^{\alpha + m_k/2}}.
\end{multline*}
From a numerical point of view, rather than using analytic tools for optimisation, we recommend plotting the values of $\operatorname{ML}_N$ versus its argument $N$, and performing graphical maximisation. This results in a computationally simple model selection procedure and we apply it in practice in Section~\ref{section:simulation}.


\section{Simulated data examples}
\label{section:simulation}

In this section we use simulations of diffusion processes with known drift and diffusion coefficient to gain insight into the numerical performance of our method. We are particularly interested in both the practical consequence of using a pseudo-likelihood ignoring the drift and the empirical rate of posterior contraction attainable in  examples. 

In the first subsection we simulate realisations from the model for different dispersion and drift coefficients. Given subsamples of those realisations sampled at different rates we
compute the posterior distribution of the dispersion function using the piecewise constant prior (histogram-type prior) with varying number of bins. As an illustration, plots of  marginal posterior bands are compared with the true dispersion function.
The marginal posterior bands are obtained by computing $1-\alpha$ central posterior intervals (see \cite{gelman}, p.~33) separately for the coefficients $\theta_k$'s using Lemma~\ref{lem:post-ig}.

By Proposition~\ref{prop1}, assuming that there is no drift still leads to consistent Bayesian estimation of the dispersion coefficient,  even if the data are from a diffusion process with nonzero drift. This is illustrated by our simulation results.

In the second subsection we use Monte Carlo methods to determine the distribution of the distance between samples from the posterior distributions and the true dispersion coefficient. In Theorem~\ref{thm:rate-posterior-ig} we showed that the posterior contraction rate in the $L_2$-norm is optimal for estimation of H\"older smooth dispersion coefficients of order $0<\lambda\leq 1$ for the prior based on the inverse gamma distribution. The results of the Monte Carlo simulation agree with this. Furthermore, we numerically determine the rate of posterior convergence in the supremum norm for two examples. The simulation results in this case are less conclusive, but suggest that the posterior contraction rate in the $L_\infty$-norm in Theorem~\ref{cor:rate:b:sup} is possibly suboptimal.

Our analyses were done employing the programming language Julia, see \cite{bezanson17}.

\subsection{Influence of the drift}

For this numerical experiment we simulated sample paths of the diffusion $(X_t:t \in [0,1])$ where the true dispersion coefficient is given by one of
\begin{align*}
s_1(t) &= 3/2 + \sin(2(4t-2)) + 2\exp(-16(4t-2)^2),\\
s_2(t) &= W_t(\omega_0) + 1,
\intertext{
and the drift is given by one of  }
b_0(x) &= 0, \\
b_1(x) &= -10x+20.
\end{align*}
The function $s_1$ is a benchmark function used in \cite{fan1996} in the context of nonparametric regression, up to a vertical shift to ensure positivity. To define $s_2(t)$, we took a fixed realisation of a Wiener path starting in $1$, with $W_t(\omega_0)  > -1$ for $t \in [0,1]$ (sampled on an equidistant grid with 800\,001 points in $[0,1]$). 
The function $s_1$ is Lipschitz continuous, while  $s_2$ is H\"older continuous with coefficient essentially $\frac12$.

Specifically, we used the Euler scheme on a grid with 800\,001 equidistant points in the interval $[0,1]$ to obtain a single diffusion path for each combination of drift and dispersion coefficients given above, which then was subsampled to obtain $n=4\,000 \cdot 2^j +1$, $j \in \{1,2, 3\}$ observations each. As the prior on the coefficients on the individual bins we used independent $\operatorname{IG(0.1, 0.1)}$ distributions.

Figures \ref{fig:s2} and \ref{fig:s4} show marginal $98\,\%$ posterior bands for different combinations of bin number and observation regime for both dispersion coefficients when the drift is zero. Figure \ref{fig:s2drift} shows the marginal posterior bands for $s_1$ which are obtained if an affine drift term $b_1(x) = -10x+20$ is present, but neglected in the estimation procedure. Comparison with Figure \ref{fig:s2} shows that presence of a strong nonzero drift
hardly affects the obtained credible bands. Note that credible bands successfully recover the overall shape of the functions $s_i$, although the recovery is not too refined; however, it would be misleading to visually compare the results obtained in the SDE setting to e.g.\ those obtainable in  nonparametric regression, as the latter constitutes a much easier inferential problem. The functions $s_i$ do not always pass through all the credible intervals, which is not surprising given the fact that these intervals are marginal. In general, construction of uniform confidence bands in nonparametric statistics is a long-studied and difficult problem, see Section 5.7 in \cite{wasserman06}; for a general perspective on nonparametric uniform confidence bands see \cite{faraway16}. Less is known about frequentist performance of nonparametric Bayesian confidence sets, although some interesting results have already been obtained in recent years, see, e.g., \cite{nickl15}, \cite{szabo15a} and \cite{szabo15b}. We do not address this issue in detail in this paper, but note that posterior contraction at an optimal rate does not automatically imply `good' frequentist coverage properties of Bayesian credible sets. Following pp.~130--131 in \cite{wasserman06} in a similar study in the case of histograms employed as nonparametric probability density estimators, in our case it is arguably more natural to consider performance of Bayesian credible bands at the resolution of the histogram-type priors, i.e.\ for a histogramised version of a dispersion coefficient $s$, obtained as
\[
\bar{s}=\sum_{k=1}^{N_n} \bar{s}_{k} \ind_{B_k}
\]
for
\[
\bar{s}_{k} = \frac{1}{\Delta_k} \int_{ B_k } s(t)\dd t, \quad k=1,\ldots,N_n,
\]
where $\Delta_k$ denotes the length of the bin $B_k$. In a frequentist approach to nonparametric inference this constitutes an alternative to attempting to get rid of the bias of a nonparametric estimator for confidence band construction purposes via an artificial device like undersmoothing. Figure \ref{fig:s2m} gives a detail of Figure \ref{fig:s2} (panels for $N=40$ and $N=160$, $n = 16\,001$, zoomed-in to show $t\in[0.2,0.5]$, posterior credible bands only) with a histogramised $s_1$ superimposed. In comparison to Figure \ref{fig:s2}, the results appear to be visually even more pleasing, with the Bayesian credible band covering the curve $s_1$ in its entirety in the left panel. This suggests to give the number of bins $N_n$ an additional interpretation of a resolution at which one is interested in learning properties of the function $s.$ Obviously, this resolution cannot be made arbitrarily fine, as this would distort the frequentist consistency property of our nonparametric Bayesian procedure as expressed in our theoretical results from Section \ref{section:asymp}. We close this brief discussion on confidence bands by mentioning the fact that a number of authors have argued in favour of the so-called average coverage as a more natural concept of coverage of confidence bands than the uniform confidence bands, see Section 5.8 in \cite{wasserman06} and references therein.

Note that the recovery is somewhat less accurate for function $s_2$, see Figure \ref{fig:s4}, than for function $s_1,$ see Figure \ref{fig:s2}. This is in perfect agreement with our theoretical results from Section \ref{section:asymp}, that give a slower posterior contraction rate for less smooth functions.

Our posterior contraction theorems only specify that the optimal number of bins $N_n$ is proportional to $n^{-\beta},$ where the exponent $\beta$ depends on the smoothness $\lambda$ of a function to be estimated. This does not give a directly applicable recipe on how to choose the proportionality constant. In practice we recommend to use our theoretical results as guidance and to try out several choices of the number of bins, cf.\ Figures \ref{fig:s2}, \ref{fig:s2drift} and \ref{fig:s4}. This is not unlike the scale-space smoothing approach in the frequentist literature, see, e.g., Section 5.11 in \cite{wasserman06}. Furthermore, a useful point of reference in our setup is the number of non-overlapping neighbouring marginal posterior intervals. Figure \ref{fig:s2m} shows that if $N$ is too small to capture adequately the curvature of a dispersion coefficient, neighbouring marginal posterior credible intervals tend to be disjoint. On the other hand, choosing too many bins leads to undersmoothing and erratic appearance of marginal credible intervals.

Numerical experiments similar to the above were also performed for other benchmark functions given in \cite{fan1996}. As the results were similar, they are not reported here.

We also tested the performance of the procedure for choosing the number of bins as discussed in Section \ref{sec:dic}. We considered the case  with  dispersion coefficient  $s_1$, drift $b_1$ and $8001$ observations. This corresponds to the leftmost column of Figure \ref{fig:s2drift}. We computed $\widehat{\operatorname{elpd}}_{\textrm{DIC}}$ for $N\in \{5, 10, 20, 40, 80, 160, 320\}$. The results are in Figure  \ref{fig:elpd}, from which it is seen that the criterion is maximised for $N = 40$.  This corresponds to the bottomleft panel of Figure \ref{fig:s2drift}. In Figure   \ref{fig:ml} we plot the results obtained with the procedure for choosing the number of bins as discussed in Section \ref{sec:bayes:factor}. Also this procedure suggests $N=40$ as an optimal number of bins.

\subsection{Empirical contraction rates}
\label{sec:emp:rate}
Of particular interest is the empirical size of the $L_2$- and $L_\infty$-balls containing most of the posterior mass. To assess this, we approximate the distribution of the $L_2$- or $L_\infty$-distance between posterior samples and the truth by sampling from the posterior. We do this for four different realisations of the model denoted by $X(\omega_1), \dots, X(\omega_4)$. Note that $q$ being the $90\,\%$-quantile of this distribution entails that the $L_2$-ball respective  $L_\infty$-ball of size $q$ contains $90\,\%$ of the posterior mass. To be specific we employ the following steps four times for each $s_i$, $i = 1,2$ and both norms.
\begin{enumerate}
\item Simulate the diffusion $(X_t(\omega))_{t \in [0,1]}$ with dispersion coefficient $s_i$ (without drift) on a grid with 800\,001 points in the interval $[0,1]$.
\item Subsample to obtain $n =2500 \cdot 2^j +1$, $j \in \{1,2, 3,4,5\}$ observations each.
\item \label{step:bins} Draw $k = 1, \dots, 2000$ samples $S^{i,n}_k$ from the posterior using $\log(N_n) = \log 5 + \frac{1}{2\lambda_i+1}\log(n)$ bins and determine the distance $\|S^{i,n} - s_i\|_2$ (respective $\log(N_n) = \log 5 + \frac{1}{2\lambda_i+1}\log(n/\log(n))$ bins for $\|S^{i,n} - s_i\|_\infty$).
\item Determine the $90\%$ quantile $q^{i,n}(\omega)$ of the distance samples and plot as function of $n$.
\end{enumerate}

Figure \ref{fig:contraction-l2}  shows the $90\%$ quantile $q^{i}(n)$ on a $\log$-$\log$ scale for the $L_2$-norm for $i = 1, 2$. The empirical findings for function $s_1$ agree very well with the exponent $\frac13 = \frac{\lambda}{ 2\lambda+1}$ for $\lambda = 1$ obtained in Theorem~\ref{thm:rate-posterior-ig}.  The function $s_2$ is $\lambda$-H\"older smooth for any $\lambda < \frac12$. The empirically determined exponent is in excellent agreement with the exponent $\frac14 = \frac{\lambda}{2\lambda+1}$ for $\lambda = \frac12$ obtained in Theorem~\ref{thm:rate-posterior-ig}.

Figure~\ref{fig:contraction-sup} shows the $90\%$ quantile $q^{i}(n)$ for the $L_\infty$-norm. The number of bins $N_n$ was chosen in analogy to nonparametric kernel regression, see Theorem 1.8 in \cite{tsyb},  as $\log(N_n) = \log 5 + \frac{1}{2\lambda_i+1}\log(n/\log(n))$. Here the results are less conclusive than in the case of the $L_2$-norm. The empirical findings suggest  the rate $(n/\log n)^\frac13$ or similar for $s_1$, and the rate $(n/\log n)^\frac14$ or similar for $s_2$.

\newcommand{\pd}{\hspace{1cm}} 
\renewcommand{\arraystretch}{0}
\newcommand{\loc}{./}
\begin{figure}
\rotatebox{90}{
\begin{tabular}{cccc}
& $N= 40$ & $N = 80$ & $N = 160$\\
\rotatebox[origin=rt]{-90}{$n=8001$\pd}&
\includegraphics[height=0.25\textwidth]{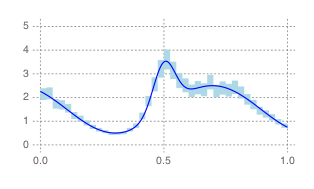} & 
\includegraphics[height=0.25\textwidth]{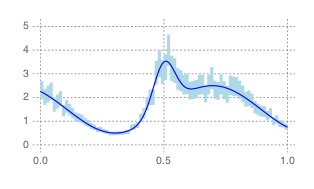} &
\includegraphics[height=0.25\textwidth]{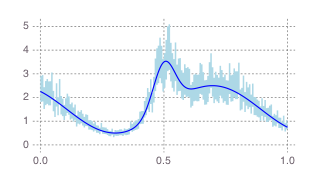} \\
\rotatebox[origin=rt]{-90}{$n=16\,001$\pd}&
\includegraphics[height=0.25\textwidth]{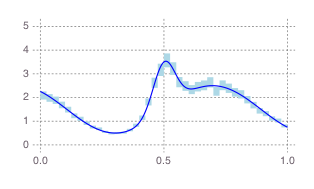}&
\includegraphics[height=0.25\textwidth]{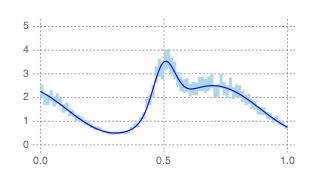}&
\includegraphics[height=0.25\textwidth]{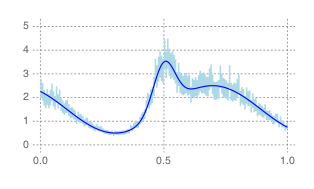}\\
\rotatebox[origin=rt]{-90}{$n=32\,001$\pd}&
\includegraphics[height=0.25\textwidth]{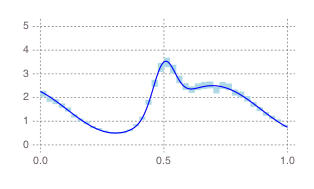}&
\includegraphics[height=0.25\textwidth]{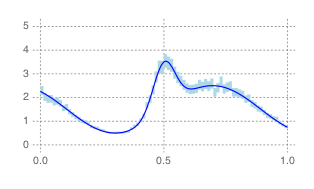}&
\includegraphics[height=0.25\textwidth]{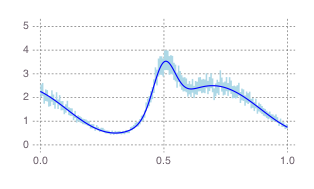}
\end{tabular}
}
\caption{Estimation results for $s_1$ with varying number of observations and bins, no drift.
The solid curve is the true dispersion coefficient while the light blue areas are $98\,\%$ marginal posterior bands. $n$ is the number of observations, $N$ is the number of bins.
}
\label{fig:s2}
\end{figure}	

\begin{figure}
\rotatebox{90}{
\begin{tabular}{cccc}
& $N= 40$ & $N = 80$ & $N = 160$\\
\rotatebox[origin=rt]{-90}{$n=8001$\pd}&
\includegraphics[height=0.25\textwidth]{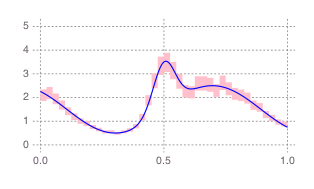} & 
\includegraphics[height=0.25\textwidth]{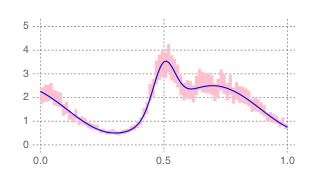} &
\includegraphics[height=0.25\textwidth]{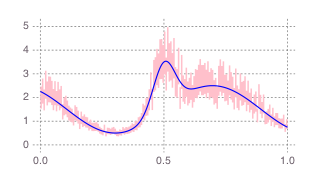} \\
\rotatebox[origin=rt]{-90}{$n=16\,001$\pd}&
\includegraphics[height=0.25\textwidth]{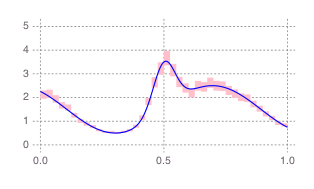}&
\includegraphics[height=0.25\textwidth]{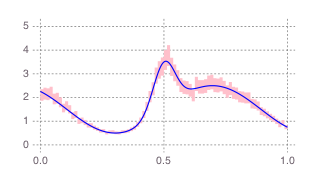}&
\includegraphics[height=0.25\textwidth]{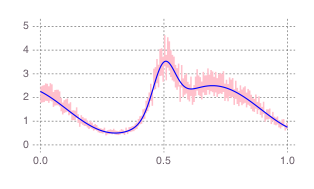}\\
\rotatebox[origin=rt]{-90}{$n=32\,001$\pd}&
\includegraphics[height=0.25\textwidth]{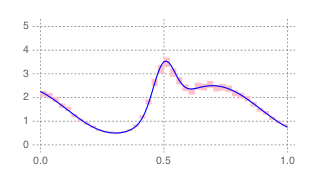}&
\includegraphics[height=0.25\textwidth]{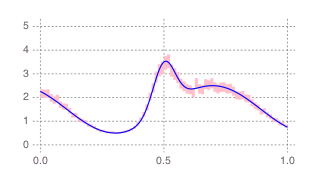}&
\includegraphics[height=0.25\textwidth]{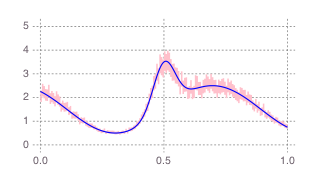}
\end{tabular}
}
\caption{Estimation for $s_1$ with varying number of observations and bins and drift $b_1(x) =-10x + 20$.
The solid curve is the true dispersion coefficient while the light red areas are $98\,\%$ marginal posterior bands. $n$ is the number of observations, $N$ is the number of bins.
}
\label{fig:s2drift}
\end{figure}	

\begin{figure}
\rotatebox{90}{
\begin{tabular}{cccc}
& $N= 40$ & $N = 80$ & $N = 160$\\
\rotatebox[origin=rt]{-90}{$n=8001$\pd}&
\includegraphics[height=0.25\textwidth]{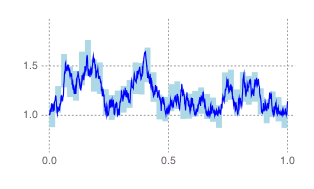} & 
\includegraphics[height=0.25\textwidth]{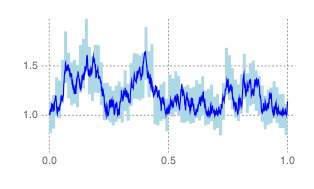} &
\includegraphics[height=0.25\textwidth]{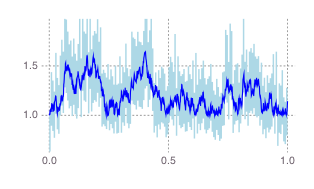} \\
\rotatebox[origin=rt]{-90}{$n=16\,001$\pd}&
\includegraphics[height=0.25\textwidth]{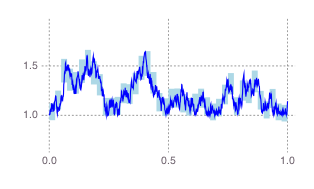}&
\includegraphics[height=0.25\textwidth]{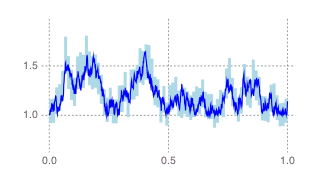}&
\includegraphics[height=0.25\textwidth]{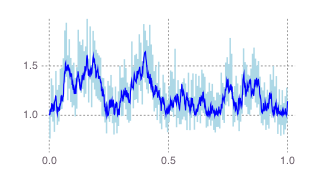}\\
\rotatebox[origin=rt]{-90}{$n=32\,001$\pd}&
\includegraphics[height=0.25\textwidth]{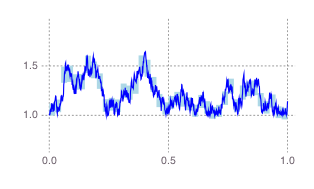}&
\includegraphics[height=0.25\textwidth]{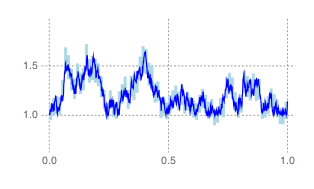}&
\includegraphics[height=0.25\textwidth]{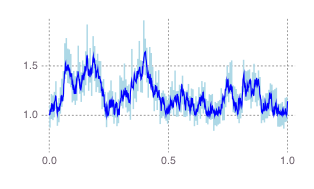}
\end{tabular}
}
\caption{Estimation results for $s_2$ with varying number of observations and bins, no drift.
The solid curve is the true dispersion coefficient while the light blue areas are 98\,\% marginal posterior bands. $n$ is the number of observations, $N$ is the number of bins.
}
\label{fig:s4}
\end{figure}	

\begin{figure}
{
\begin{tabular}{ccc}
 $N = 40$ & $N = 160$\\
\includegraphics[width=0.45\textwidth]{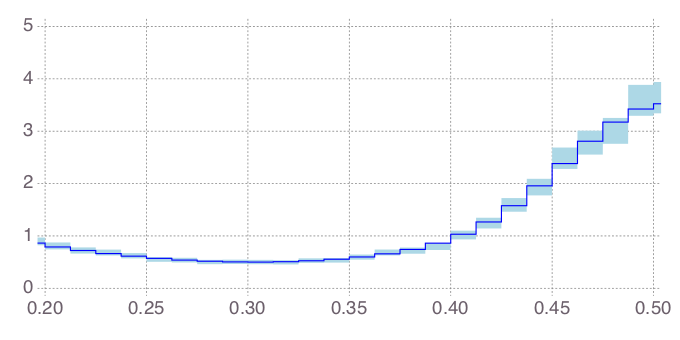}&
\includegraphics[width=0.45\textwidth]{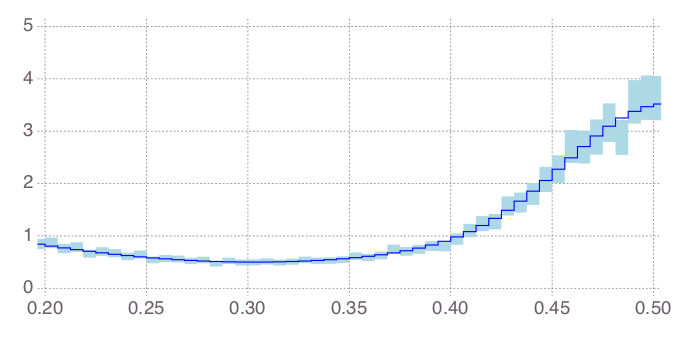}\\
\end{tabular}
}
\caption{Detail of Figure \ref{fig:s2}, panels for $N=40$ and $N=160$, $n = 16\,001$, zoomed-in to show $t\in[0.2,0.5]$, posterior credible bands only, with a histogramised $s_1$ superimposed.}
\label{fig:s2m}
\end{figure}	

\begin{figure}[htbp]
\begin{center}
\includegraphics[width=0.9\textwidth]{\loc 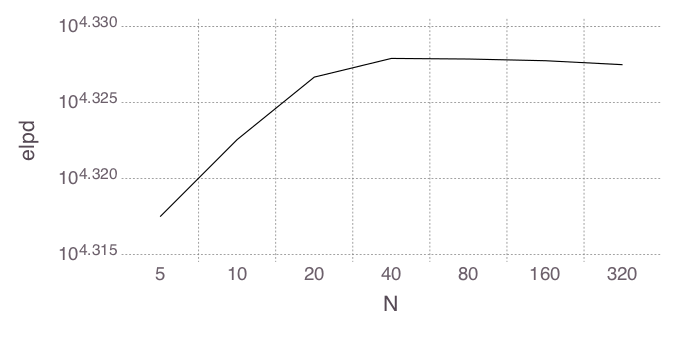}
\caption{A log plot of the $\widehat{\operatorname{elpd}}_{\textrm{DIC}}$-values for estimating the dispersion coefficient $s_1$ versus different bin numbers $N$, $n$ fixed at $8001$, with drift $b_1$.
 }
\label{fig:elpd}
\end{center}
\end{figure}

\begin{figure}[htbp]
\begin{center}
\includegraphics[width=0.9\textwidth]{\loc 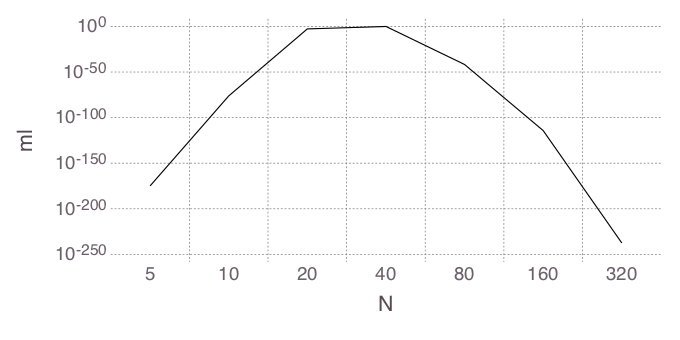}
\caption{A log plot of suitably scaled and shifted $\operatorname{ML}$-values for estimating the dispersion coefficient $s_1$ versus different bin numbers $N$, $n$ fixed at $8001$, with drift $b_1$.
 }
\label{fig:ml}
\end{center}
\end{figure}

\begin{figure}
\includegraphics[width=0.9\textwidth]{\loc 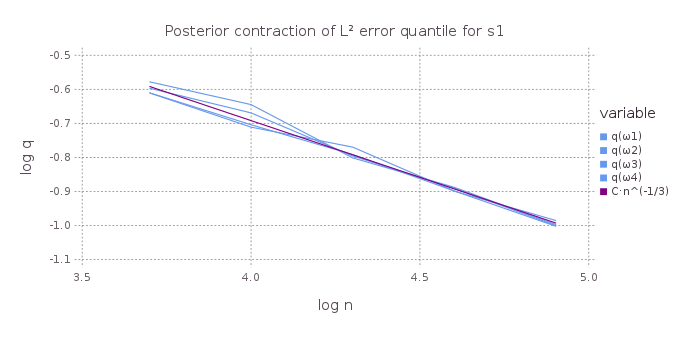}
\includegraphics[width=0.9\textwidth]{\loc 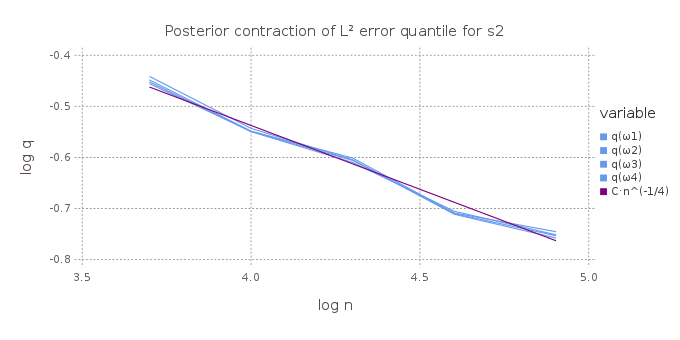}
\caption{Estimate of the contraction rate of the size of $L_2$-balls around $s$ covering 90 $\%$ of posterior probability mass by sampling (blue); compared with rate (purple). Top: Lipschitz continuous case, example $s_1$. Bottom: H\"older coefficient essentially $\frac12$, example $s_2$.}
\label{fig:contraction-l2}
\end{figure}

\begin{figure}
\includegraphics[width=0.9\textwidth]{\loc 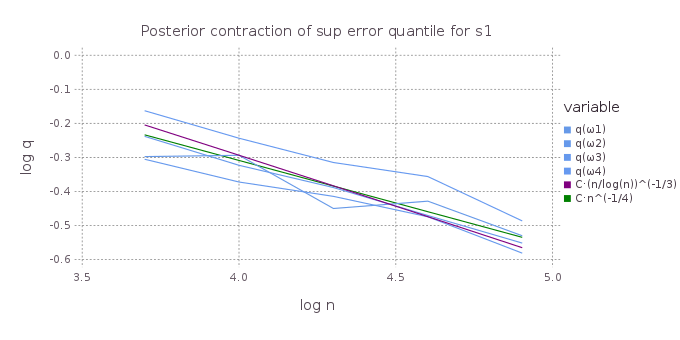}
\includegraphics[width=0.9\textwidth]{\loc 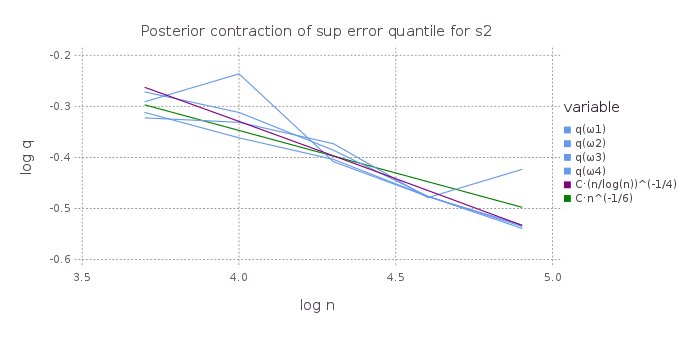}
\caption{Estimate of the contraction rate of the size of $L_\infty$-balls around $s$ covering 90 $\%$ of posterior probability mass by sampling (blue); compared with likely rate (purple) and rate of Theorem~\ref{cor:rate:b:sup} (green). Top: Lipschitz continuous case, example $s_1$. Bottom: H\"older coefficient essentially $\frac12$, example $s_2$.}
\label{fig:contraction-sup}
\end{figure}

\section{Real data examples}\label{section:realdata}
In this section we apply our Bayesian method on two real data examples and study its implications. The daily exchange rates (noon buying rates in New York City for cable transfers payable in foreign currencies)
 JPY/USD and USD/GBP from January 1, 1999, to March 20, 2010
are available as data sets DEXJPUS and DEXUSUK from \cite{fred}.
We visualise the data in Figure~\ref{fig:exchangedata} (time in this and subsequent figures corresponds to the physical time, with unit being a year).

These exchange rate time series were considered e.g.\ in \cite{hamrick2011}. 
Based on discrete time observations assumed to have arisen from the solution of an SDE
\begin{equation}
\label{ham}
\dd X_t=b(X_t)\,\dd t+\sigma(X_t)\,\dd W_t, \quad X_0=x, \quad t\in[0,T],
\end{equation}
with space-dependent dispersion coefficient $\sigma,$ \cite{hamrick2011} proposed a maximum penalised quasi-likelihood method to estimate nonparametrically the diffusion coefficient $\sigma^2.$

Plots of both series appear to indicate that the data are nonstationary, and this is confirmed also by the outcomes of the augmented Dickey-Fuller test we performed using {\tt urca} package in {\bf R}, see \cite{rcore17} and \cite{urca}. The test constitutes a standard unit root test in time series analysis, see, e.g., Chapter 17 in \cite{hamilton} and Section 5.3 in \cite{aragon} for additional information. A similar conclusion on nonstationarity is reached in \cite{hamrick2009}.
As stationarity is no prerequisite for application of our nonparametric Bayesian method, we do not pursue this question any further in the paper. We retrieved the estimates $\hat \sigma$ given in \cite{hamrick2011} from the figures published in the digital version of the publication using WebPlotDigitizer, see \cite{rohatgi}, and next used them to calculate the induced estimates $t \mapsto \hat s(t) = \hat\sigma(X_t)$ of the historical volatility at time $t$. In Figures \ref{fig:JPUS} and \ref{fig:USUK} we contrast those induced estimates $\hat s$ with $90\%$ marginal posterior bands for the deterministic dispersion coefficient $s_0$ that were obtained through our Bayesian procedure. Since nominal exchange rates vary widely with the denomination used, we employed a non-informative $\operatorname{IG(0.001, 0.001)}$ prior for coefficients $\theta_k$'s from Lemma \ref{lem:post-ig}. Our estimation results in Figures \ref{fig:JPUS} and \ref{fig:USUK} show that the volatility was high in the final years of  the decade 2000--2010, coinciding with the sub-prime mortgage crisis and the following recession. On the other hand, the model from \cite{hamrick2011} does not appear to capture this fact. It is reassuring to see that our method recovers this relevant event from the data. Based on this observation, we believe that including time-dependence of the volatility into the model is more appropriate in this example.

\begin{figure}
\includegraphics[width=0.95\textwidth]{\loc 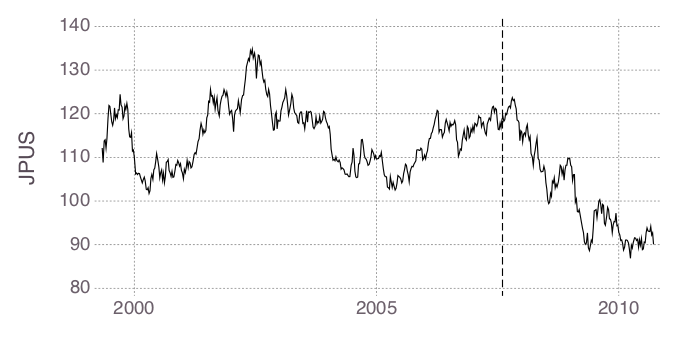}\\
\includegraphics[width=0.95\textwidth]{\loc 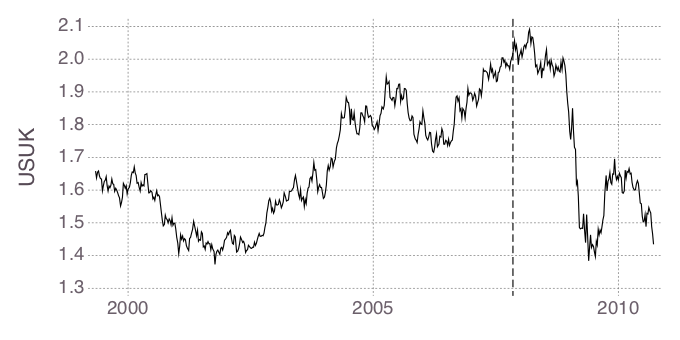}
\caption{Daily exchange rates JPY/USD (top) and USD/GBP (bottom) from January 1, 1999 to March 20, 2010.
Change-point estimates  are indicated with dashed vertical lines. 
}
\label{fig:exchangedata}
\end{figure}

\begin{figure}
\includegraphics[width=0.9\textwidth]{\loc 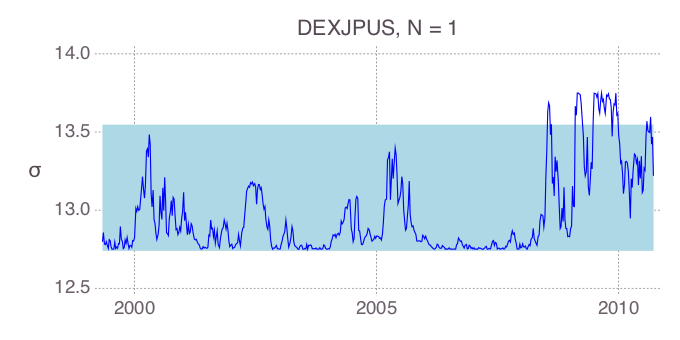}
\includegraphics[width=0.9\textwidth]{\loc 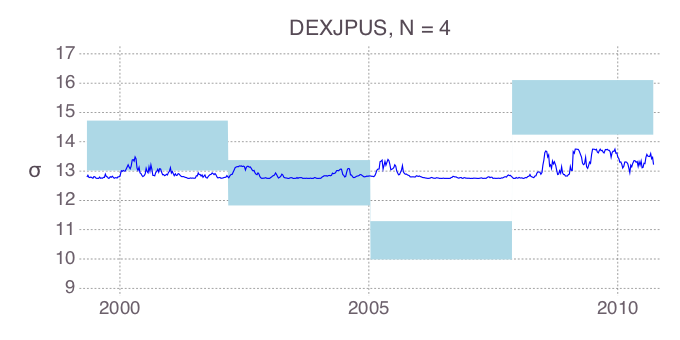}
\includegraphics[width=0.9\textwidth]{\loc 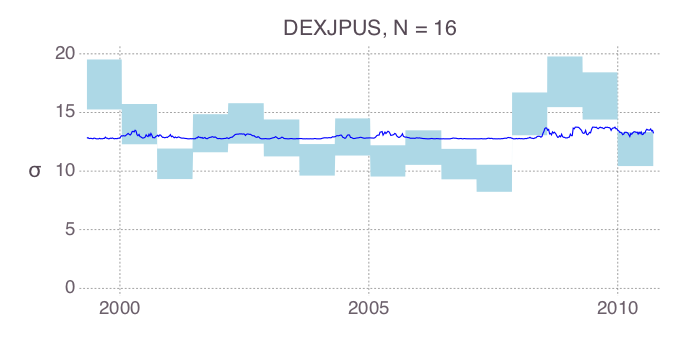}
\caption{$90\%$ marginal posterior band for the volatility of DEXJPUS. Plot of $t \mapsto \hat \sigma(X_t)$ as induced by the estimates in \protect\cite{hamrick2011}. 
}
\label{fig:JPUS}
\end{figure}

\begin{figure}
\includegraphics[width=0.9\textwidth]{\loc 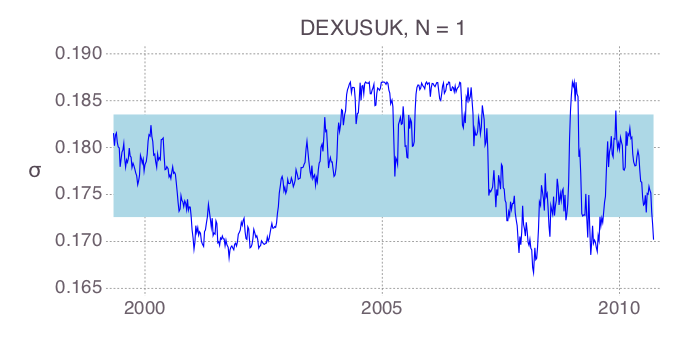}
\includegraphics[width=0.9\textwidth]{\loc 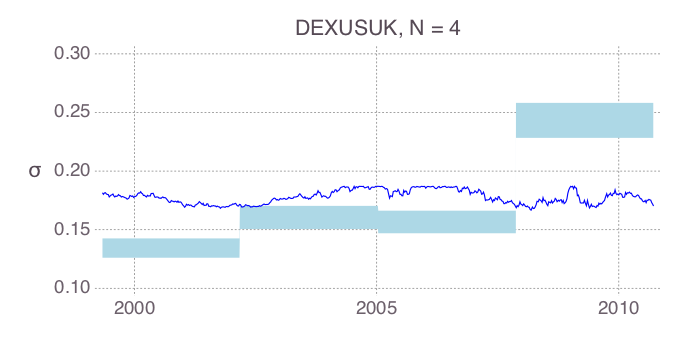}
\includegraphics[width=0.9\textwidth]{\loc 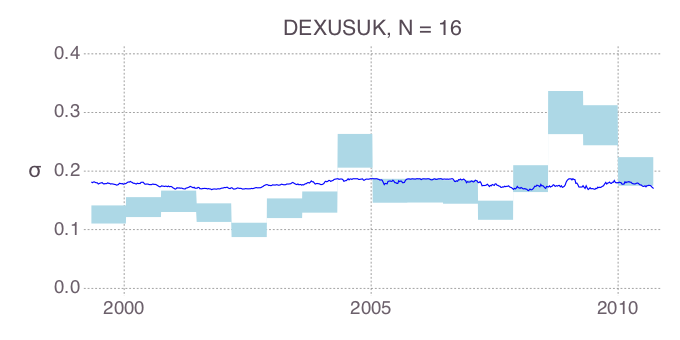}
\caption{$90\%$ marginal posterior band for volatility of DEXUSUK. Plot of $t \mapsto \hat \sigma(X_t)$. 
}
\label{fig:USUK}
\end{figure}

A further confirmation of our findings comes from the change-point analysis of the data. For both the DEXJPUS and DEXUSUK series, a simple estimator for detection of a change-point in diffusivity of an SDE, see \cite{degregorio08} and Section 4.3.1 in \cite{iacus}, that is implemented in {\bf R} in the {\tt sde} package (see \cite{sde}), yields a change-point in the volatility level before and after 2007, with volatility prior to 2007 being lower. This is in excellent agreement with findings using our Bayesian method. See Figure \ref{fig:exchangedata} below, where we plot the exchange rate data together with change point estimates, and compare to Figures \ref{fig:JPUS} and \ref{fig:USUK}. We note, however, that our Bayesian approach yields more, in that the method from \cite{degregorio08} assumes that the volatility is constant before and after the change-point, which is not what our Bayesian marginal credible sets suggest.

\section{Proofs}
\label{section:proofsgeneral}

\begin{proof}[Proof of Proposition~\ref{prop1}.]
By Theorem 6.10 in \cite{hopfner} and our Assumption \ref{standing}~(b)--(c), the laws $\pp$ and $\mathbb{P}_{0,s_0}$ of the path $(X_t:t\in[0,1])$ are equivalent, the result that ultimately relies upon Girsanov's theorem. 
Let $Z=\frac{\dd\pp}{\dd\mathbb{P}_{0,s_0}}$, and let $Z^n=\frac{\dd\mathbb{P}^n_{b_0,s_0}}{\dd\mathbb{P}^n_{0,s_0}}$ be the density of the respective laws of $\mathcal{X}_n$. By Theorem 2 on p.~245 in~\cite{skorohod64} and Corollary 2 on p.~246 there, $Z^n=\mathbb{E}_{{0,s_0}}(Z\mid \mathcal{X}_n ).$ 
Then convergence of $\Pi_n(U_{s_0,\varepsilon_n}^c\mid \mathcal{X}_n)$ in $\mathbb{P}_{0,s_0}$-probability to zero implies that it also converges to zero in $\pp$-probability. Indeed, fix $\eta>0$, let $A^n=\{\Pi_n(U_{s_0,\varepsilon_n}^c\mid \mathcal{X}_n)>\eta\}$ and let $\eps>0$. Then 
\[\pp(A_n)=\mathbb{E}_{{0,s_0}}[Z^n\mathbf{1}_{A^n}]=\mathbb{E}_{{0,s_0}}[\mathbb{E}_{{0,s_0}}(Z\mathbf{1}_{A^n}\mid  \mathcal{X}_n)]=\mathbb{E}_{{0,s_0}}[Z\mathbf{1}_{A^n}].\] Choose $\delta>0$ such that $\mathbb{P}_{0,s_0}(A)<\delta$ for any event $A$ implies $\pp(A)=\mathbb{E}_{{0,s_0}}[Z\mathbf{1}_{A}]<\eps$, possible in view of Lemma~13.1 in \cite{williams}. As eventually $\mathbb{P}_{0,s_0}(A^n)<\delta$, we have $\pp(A_n)<\eps$.
\end{proof}

\begin{lemma}\label{thm:rate}
Define 
\[ 
M(x):=\expi (s^2(x) \mid  \scr{X}_n),
\] 
the posterior mean of $s^2(x)$. Assume $m_n\asymp n^{2\lambda/(2\lambda+1)}$ (equivalently, $N_n \asymp n^{1/(2\lambda+1)}$).  If we let  $ \eps_n\asymp m_n^{-1/2}$, i.e.\ $\eps_n\asymp n^{-\lambda/(2\lambda+1)}$, then for any sequence $h_n$ tending to infinity we have for any fixed $x\in [0,1]$
\[ \ppz\left(|M(x)-s_0^2(x)| \ge \eps_n h_n\right)  \to 0 \]
as $n\to \infty.$
\end{lemma}

\begin{proof}
Assume $x\in B_k$ for some  $k<N$ (note that then $k=\lfloor \frac{nx}{m}\rfloor+1$). The case $x\in B_N$ follows later on. We compute
\begin{align*} M(x) &=\expi(\th_k \mid \scr{X}_n) = \frac{\beta+\frac{n}{2}\sum_{i=(k-1)m+1}^{km}Y_{i,n}^2}{\alpha+m/2-1}\\ &=\frac{2\beta}{2(\alpha-1)+m}+\frac{n\sum_{i=(k-1)m+1}^{km}Y_{i,n}^2}{2(\alpha-1)+m}.\end{align*} 
Note that
\[
\exz \left[ \sum_{i=(k-1)m+1}^{km}Y_{i,n}^2\right]=\int_{t_{m(k-1)}}^{t_{mk}}s_0^2(u)\,\dd u.
\]
Hence,
\[
b(x):=\exz [M(x)]-s_0^2(x)=\frac{2\beta}{2(\alpha-1)+m}+\frac{n\int_{t_{m(k-1)}}^{t_{mk}}s_0^2(u)\,\dd u}{2(\alpha-1)+m}-s_0^2(x).
\]
We consider
\begin{align*}
\frac{n\int_{t_{m(k-1)}}^{t_{mk}}s_0^2(u)\,\dd u}{2(\alpha-1)+m}-s_0^2(x) 
& =\frac{n\int_{t_{m(k-1)}}^{t_{mk}}s_0^2(u)\,\dd u}{2(\alpha-1)+m}-\frac{n}{m}\int_{t_{m(k-1)}}^{t_{mk}}s_0^2(x)\,\dd u 
\end{align*}
\[
= \frac{n}{m+2(\alpha-1)}\int_{t_{m(k-1)}}^{t_{mk}}(s_0^2(u)-s_0^2(x))\,\dd u -\frac{2(\alpha-1)s_0^2(x)}{m+2(\alpha-1)}.
\]
As $s_0$ is bounded from above by some constant $\mathcal{K}>0$, the last term is of order $\frac{1}{m}$. We continue with the integral expression. By H\"older continuity of $s_0$, it follows that  $|s_0^2(u)-s_0^2(v)|\leq 2\mathcal{K}L|u-v|^\lambda$.
Using the sharp bound (attained at $x=t_{m(k-1)}$ and $x=t_{mk}$)
\[
\int_{t_{m(k-1)}}^{t_{mk}}|u-x|^\lambda\,\dd u \leq \frac{1}{\lambda+1}\left(\frac{m}{n}\right)^{\lambda+1},
\]
we get the order bound (uniformly in $x \in B_k$)
\[
|b(x)|\leq O\left(\frac{1}{m}\right)+O\left(\frac{n}{m}\left(\frac{m}{n}\right)^{\lambda+1}\right)=O\left(\frac{1}{m}\right)+O\left(\frac{m}{n}\right)^{\lambda}.
\]
For the variance we obtain (using that $Y_{i,n}$ and $Y_{j,n}$ are independent for $i\neq j$)
\begin{align*} \varz [M(x)] &= \left(\frac{n}{2(\alpha-1)+m}\right)^2 \sum_{i=(k-1)m+1}^{km} \varz [Y_{i,n}^2]	\\ & =
\left(\frac{n}{2(\alpha-1)+m}\right)^2 2\sum_{i=(k-1)m+1}^{km}\left(\int_{t_{i-1}}^{t_i} s_0^2(u) d u \right)^2 \\&\le C_\alpha \scr{K}^4\frac{1}{m} = O\left(\frac1{m}\right),
\end{align*}
for a positive constant $C_\alpha$, uniformly for $x\in B_k$. 
Balancing bias and standard deviation yields the order equality
\[
O\left(\frac{1}{m}\right)+O\left(\frac{m}{n}\right)^{\lambda}=O\left(\frac{1}{\sqrt{m}}\right),
\]
or
\[
O(1)+O\left(\frac{m^{\lambda+1}}{n^{\lambda}}\right)=O(\sqrt{m}),
\]
which gives $m^{\lambda+1/2}\asymp n^{\lambda}$, so that
\begin{equation}\label{eq:N-m}
m\asymp n^{\frac{2\lambda}{2\lambda+1}}.
\end{equation}

A similar analysis holds for $x\in B_N$. We highlight the main steps for this case. For instance, we now get 
\[
M(x) =\expi(\th_N | \scr{X}_n) =\frac{2\beta}{2(\alpha-1)+m+r}+\frac{nZ_N}{2(\alpha-1)+m+r},
\]
and
\[
\exz [Z_N]=\int_{t_{m(N-1)}}^1s_0^2(u)\,\dd u.
\]
This results in the bias
\[
b(x)=O\left(\frac{1}{m}\right) + O\left(\frac{m+r}{n}\right)^{\lambda}=O\left(\frac{1}{m}\right) + O\left(\frac{m}{n}\right)^{\lambda},
\]
as $0\leq r<m$.
Similarly, we get
\[
\varz [M(x)] =  O\left(\frac1{m+r}\right),
\]
which is, using again $0\leq r<m$, of order $O\left(\frac1{m}\right)$.
Hence balancing gives the same order relation $m\asymp n^{\frac{2\lambda}{2\lambda+1}}$ as for the case $k<N$ in \eqref{eq:N-m}.

With the above choice for $m$ we obtain, for any $x\in [0,1]$, for the mean squared error that
\[  \exz  [\left(M(x) -s_0^2(x)\right)^2]  \asymp n^{-\frac{2\la}{2\la+1}}. \]
Hence, upon taking 
$
\eps_n\asymp n^{-\frac{\lambda}{2\lambda+1}}, 
$ the result follows from  Chebyshev's inequality:
\[ \ppz\left(|M(x)-s_0^2(x))| \ge \eps_n h_n\right) \le \frac{ \exz [\left(M(x) -s_0^2(x)\right)^2]  }{\eps_n^2 h_n^2} \lesssim \frac1{h_n^2} \to 0\]
as $n\to \infty.$
\end{proof}

\begin{cor}\label{cor:rate}
Under the assumptions of Lemma \ref{thm:rate} we  have
\[ \exz [ \|M-s_0^2\|_2^2] \asymp \eps_n^2, \]
where $\|\cdot\|_2$ is the $L_2$-norm on $[0,1]$. 
\end{cor}

\begin{proof}
Simply note that 
\begin{align*} {\exz [\|M-s_0\|^2_2]}
& =  {\int_0^1 \exz [|M(x)-s_0(x)|^2]\,\dd x}\\ & =  \sum_{k=1}^{N_n} \int_{B_k} \exz [|M(x)-s_0(x)|^2]\, \dd x,
\end{align*}
and apply the bounds derived in the proof of Lemma \ref{thm:rate} on each bin $B_k$ separately, $1\leq k\leq N_n$. 
\end{proof}

\begin{proof}[Proof of Theorem \ref{thm:rate-posterior-ig}]
First assume $b_0\equiv 0$. 
By Chebyshev's inequality, one has
\begin{equation}\label{eq:bv0}
\exz [\Pi_n(\|s^2-s_0^2\|_2 \ge h_n \eps_n \mid \scr{X}_n)] \le (h_n \eps_n)^{-2}\exz\left[\expi \left(\|s^2-s_0^2\|_2^2 \mid \mathcal{X}_n\right)\right].
\end{equation}
The bias-variance decomposition gives
\begin{equation*}\label{eq:bv1}
\exz\left[\expi (\|s^2-s_0^2\|_2^2|\mathcal{X}_n)\right] =\exz\left [ \|M-s_0^2\|_2^2\right] +\sum_{k=1}^{N_n}\int_{B_k}\exz[\vpi (s^2(x) \mid \mathcal{X}_n)]\dd x.
\end{equation*}
The behaviour of the  first term on the righthand side was derived in the proof of Lemma \ref{thm:rate}. For the second term, we consider $x\in B_k$ for $k< N$. As in the proof of Lemma \ref{thm:rate}, the case $x\in B_N$ is similar. We obtain
\begin{equation}
\label{eq:bv2}
\begin{split}
\vpi (s^2(x)|\mathcal{X}_n)&=\vpi (\th_k|\mathcal{X}_n) = \frac{(\beta+nZ_k/2)^2}{(\alpha+m/2-1)^2(\alpha+m/2-2)}\\
&= O\left(\frac1{m^3}\right) + O\left(\frac{n^2}{m^3}\right) Z_k^2.
\end{split}
\end{equation}
By Assumption~\ref{standing}~(c), $s_0$ is bounded by a positive constant $\scr{K}$. 
We have 
\begin{align*} \exz [Z_k^2] &= \varz [Z_k] + \left( \exz [Z_k]\right)^2=\sum_{i=(k-1)m+1}^{km} \varz [Y_{i,n}^2] + \left(\int_{B_k} s_0^2(u)\,\dd u \right)^2\\ &
 \le  2 \sum_{i=(k-1)m+1}^{km} \left(\int_{t_{i-1}}^{t_i} s^2(u)\,\dd u \right)^2 + \left(\frac{m}{n}\right)^2 \scr{K}^4 \\ &
 \le 2\sum_{i=(k-1)m+1}^{km} \scr{K}^4 n^{-2} + \left(\frac{m}{n}\right)^2 \scr{K}^4  =O\left(\frac{m}{n^2}\right)+O\left(\frac{m}{n}\right)^2=O\left(\frac{m}{n}\right)^2.
\end{align*}
This implies 
\begin{equation*}\label{eq:bv3}
\exz\left[\vpi (s^2(x)|\mathcal{X}_n)\right]=O\left(\frac1{m^3}\right) +O\left(\frac{n^2 }{m^3} \frac{m^2}{n^2}\right)=O\left(\frac1{m}\right).\end{equation*}
Combining the above order bounds then yields  
\[ \exz \left[\expi (\|s^2-s_0^2\|^2|\mathcal{X}_n) \right]  = \left[ O\left(\frac{1}{m}\right)+O\left(\frac{m}{n}\right)^{\lambda}\right]^2 + O\left(\frac1{m}\right) .\]
Balancing these terms was already done in the proof of Lemma \ref{thm:rate}, and gives the value for $m$ depending on $n$ as in display \eqref{eq:N-m}. This implies the stated result for $b_0\equiv 0$. 
Proposition \ref{prop1} implies the result is not only true for $b_0\equiv 0$, but in general.
\end{proof}

\begin{proof}[Proof of Theorem \ref{cor:rate:b:sup}]
We first prove
\begin{equation}\label{eq:statement1}
	\exz \left[ \sup_{x\in [0,1]}|M(x)-s_0^2(x)| \right] \asymp \tilde\eps_n.
\end{equation}  
For that, we first consider a selected bin $B_k$ for $k<N$ (the analysis on $B_N$ would yield the same order estimates) and $\sup_{x\in B_k}|M(x)-s_0^2(x)|$. Note that for $x\in B_k$ one has that $M(x)$ does not explicitly  depend on $x,$ being equal to $\exz(\theta_k|\mathcal{X}_n).$ We thus occasionally write $M(x)=M_k$, whenever this is convenient.
For any $x\in B_k$ one has, recalling $b(x)=\exz [M(x)]-s_0^2(x)$,
\begin{align*}
|M(x)-s_0^2(x)|^2 & \leq 2|M(x)-\exz [M(x)]|^2 + 2|\exz [M(x)]-s_0^2(x)|^2 \\
& = 2|M_k-\exz [M_k]|^2 + 2|b(x)|^2, 
\end{align*}
and hence
\begin{align*}
\sup_{x\in B_k}|M(x)-s_0^2(x)|^2 & \leq 2|M_k-\exz [M_k]|^2 + 2\sup_{x\in B_k}|b(x)|^2. 
\end{align*}
Turning to $\sup_{x\in [0,1]}|M(x)-s_0^2(x)|$, we obtain from the above that
\begin{equation}
\label{rate:loss}
\begin{split}
\sup_{x\in [0,1]}|M(x)-s_0^2(x)|^2 & = \sup_k \sup_{x\in B_k}|M(x)-s_0^2(x)|^2 \\
& \leq 2\sup_k\, |M_k-\exz [M_k]|^2 + 2\sup_k\sup_{x\in B_k}|b(x)|^2 \\
& \leq 2\sum_k\, |M_k-\exz [M_k]|^2 +2\sup_k\sup_{x\in B_k}|b(x)|^2,
\end{split}
\end{equation}
and therefore
\[
\exz\left[\sup_{x\in [0,1]}|M(x)-s_0^2(x)|^2\right]  \leq 2\sum_k\varz [M_k]  + 2\sup_k\sup_{x\in B_k}|b(x)|^2. 
\]
Note that in the proof of Lemma~\ref{thm:rate} we obtained the uniform order bounds,  not depending on $x$ and $k$,
$b(x)=O\left(\frac{1}{m}\right) + O\left(\frac{m}{n}\right)^{\lambda}$ and 
$\varz [M(x)] = O\left(\frac1{m}\right)$.
It follows that 
\begin{align*}
\exz \left[\sup_{x\in [0,1]}|M(x)-s_0^2(x)|^2\right]
& \leq N_n O\left(\frac{1}{m}\right) + \left(O\left(\frac{1}{m}\right)+O\left(\frac{m}{n}\right)^{\lambda}\right)^2 \\
& \leq O\left(\frac{n}{m^2}\right)  + \left(O\left(\frac{1}{m}\right)+O\left(\frac{m}{n}\right)^{\lambda}\right)^2.
\end{align*}
Balancing of the two summands, using $m=O(n^{\alpha})$, is obtained for $\alpha=\frac{2\lambda+1}{2\lambda+2}$. This results in 
\[
\exz\left[\sup_{x\in [0,1]}|M(x)-s_0^2(x)|^2\right] = O(n^{-\frac{\lambda}{\lambda+1}}),
\]
so that
\[
\exz\left[\sup_{x\in [0,1]}|M(x)-s_0^2(x)| \right]= O(n^{-\frac{\lambda}{2\lambda+2}}),
\]
which completes the proof of \eqref{eq:statement1}. 
The proof of the second statement of the theorem follows upon combining elements from the proofs of Theorems~\ref{thm:rate-posterior-ig} and of this result. 

Proposition \ref{prop1} implies the result is not only true for $b_0\equiv 0$, but in general.
\end{proof}

\appendix
\section{}
\label{appendix:asymp}

In this appendix we present an alternative asymptotic frequentist analysis of our Bayesian procedure. A principal difference with the one in the main text is that our prior on the coefficients $\theta_k$'s is not bound to be inverse gamma. On the downside, the $L_2$-posterior contraction rate we obtain is slower than that in Theorem~\ref{thm:rate-posterior-ig}. This is due to the fact that in our arguments we cannot rely on the conjugacy of the inverse gamma prior anymore.

\begin{defin}
\label{classX} Let $\mathcal{S}_n$ denote a set of dispersion coefficients $s\colon[0,1]\rightarrow[\kappa,\mathcal{K}],$ that are piecewise constant on the bins $B_k$:
\begin{equation}
\label{s1}
s=\sum_{k=1}^{N_n} \xi_k \ind_{B_k}.
\end{equation}
\end{defin}
The prior $\Pi_n$ on the dispersion coefficient $s$ is defined by putting a prior on the coefficients $\xi_k$'s. 

\begin{equation*}
s^2=\sum_{k=1}^{N_n} \xi_k^2 \ind_{B_k}=\sum_{k=1}^{N_n} \theta_k \ind_{B_k},
\end{equation*}
where we have set $\theta_k=\xi_k^2.$

\subsection{General contraction rate}\label{appendix:general}

\begin{thm}
\label{mainthm}
Let Assumption \ref{standing} hold with bounds $0<\kappa\leq s_0(t)\leq\mathcal{K}<\infty$ for all $t\in[0,1]$.
Assume the prior $\Pi_n$ is defined as the law of a random function $s$ from \eqref{s1}, where  the random variables $\kappa\leq \xi_k\leq \mathcal{K},$ $k=1,\ldots,N_n,$ are independent and identically distributed  with a density that is bounded away from zero on the interval $[\kappa,\mathcal{K}].$  
Then for any sequence $m_n\asymp n^{1-\alpha}$, equivalently $N_n\asymp n^{\alpha}$, with $\alpha \in [\frac{1}{2},1-\frac{\lambda}{2}]$,  
there exists a constant $\widetilde{M}>0$, such that for $\varepsilon_n= \widetilde{M}n^{-\beta}\log^\gamma n$  with $\beta=\frac{\lambda}{4}$ and arbitrary $\gamma>1$,
\rm
\begin{equation}\label{eq:post}
\ee_{b_0,s_0}[\Pi_n(U_{s_0,\varepsilon_n}^c|\scr{X}_n)]\rightarrow 0
\end{equation}
as $n\rightarrow\infty.$
\end{thm}

The essential term determining the posterior contraction rate is $n^{-\beta}$. Here $\beta$ depends on $\lambda$ in an increasing way: a smoother function $s_0$ allows for a faster contraction rate. The optimality of the choice of $\alpha$, the exponent $\beta$ and the condition $\gamma>1$, at least within the context of our proof,
is elaborated in Remark~\ref{rem:optimal} on page \pageref{rem:optimal}. The best possible rate obtainable from Theorem \ref{mainthm} is achieved for $\lambda = 1$ and is (essentially) $n^{-1/4}$. The possibility that the posterior contraction rate for arbitrary priors is in fact faster than the one given in Theorem \ref{mainthm} cannot be discarded: although our proofs for this general result use intricate technical arguments, they still might be not sharp enough. 
\rm

\begin{rem}
The statement of Theorem \ref{mainthm} is in some respect similar to that of Theorem 1 in \cite{gugu16}. Significant differences are that \cite{gugu16} choose a different prior and assume that the function $s_0$ is differentiable; cf.~our remarks in the introduction to the paper. Not only does this amount to difference in statements, but also proofs in the present paper are more involved than the ones in \cite{gugu16}. 
\end{rem}

\begin{rem}\label{rem:l2}
The inequality
\begin{equation*}
\|s-s_0\|_2 \geq \frac{1}{2\mathcal{K}} \|s^2-s_0^2\|_2,
\end{equation*}
valid for $s\in\mathcal{S}_n$ and $s_0$ satisfying Assumption \ref{standing}, together with Theorem \ref{mainthm} implies that also the posterior for the diffusion coefficient $s_0^2$ contracts around the truth with the rate $\varepsilon_n$ given in that theorem. 
\end{rem}

\subsection{Proof of Theorem~\ref{mainthm}}
\label{appendix:proofsgeneral}

Some of by now classical references, where general conditions for derivation of posterior contraction rates are given, include \cite{ghosal00}, \cite{ghosal07} and \cite{shen01}. The proof of Theorem \ref{mainthm} follows the same roadmap as these papers, however without appealing directly to their results, as our statistical setup is somewhat different from the ones covered by those papers. In particular, note that the distribution of the observations $X_{t_{i,n}}\!$'s depends on the index $n,$ which is not covered by the results in the above-mentioned references.

We first introduce some notation and definitions: $p_{i,n,s}$ and $p_{i,n,0}$ will be the densities of increments $Y_{i,n}=X_{t_{i,n}}-X_{t_{i-1,n}}$ under the parameter values $s$ and $s_0,$ with drift equal to $b_0=0$ in both cases. The notation $Z_{i,n,s}(Y_{i,n})=\log (p_{i,n,s}(Y_{i,n})/p_{i,n,0}(Y_{i,n}))$ will stand for the log-likelihood ratio corresponding to a single observation $Y_{i,n}.$ $R_n(s)=L_n(s)/L_n(s_0)$ will denote the likelihood ratio. Furthermore, in line with the notation in \cite{geer00}, we let
\begin{equation*}
z_i=t_{i-1,n},\quad\mathcal{W}_i=1-\frac{Y_{i,n}^2}{\int_{t_{i-1,n}}^{t_{i,n}}s_0^2(u)\,\dd u}, \quad f_s(z)=\frac{\int_{z}^{z+1/n}[s_0^2(u)-s^2(u)]\,\dd u}{\int_{z}^{z+1/n}s^2(u)\,\dd u}.
\end{equation*}
Under the parameter pair $(0,s_0),$ the random variables $\mathcal{W}_i$'s are i.i.d.\ with zero mean and variance equal to two. Since their distributions do not depend on $n,$ we take the liberty to omit an extra index $n$ in our notation. For notational simplicity, we also omit an extra index $n$ in $z_i$'s and $f_s,$ though, strictly speaking, it is still required there.

\begin{proof}[Proof of Theorem \ref{mainthm}]
As in the proofs of results from the main body of the paper, we may assume $b_0=0$: the statement for a general $b_0$ follows from this particular case, see Proposition~\ref{prop1}. The general structure of the proof is similar to the one of Theorem 1 in \cite{gugu16} and ultimately \cite{ghosal00} and \cite{shen01}, but many details differ, as evidenced in particular by the proofs of the lemmas in Appendix~\ref{section:proofsremaining}.

Write
\begin{equation*}
\Pi_n({U}_{s_0,\varepsilon_n}^c\mid X_{t_{0,n}}\ldots,X_{t_{n,n}})
=\frac{\int_{{U}_{s_0,\varepsilon_n}^c} R_n(s) \Pi(\dd s)}{ \int_{\mathcal{S}} R_n(s) \Pi_n(\dd s) }=\frac{\mathcal{N}_n}{\mathcal{D}_n}.
\end{equation*}
Let $\epsilon>0$. For any events $E_n$ and $F_n$ we have
\begin{equation}\label{eq:ab}
\mathbb{P}_{0,s_0}\left(\frac{\mathcal{N}_n}{\mathcal{D}_n}>\epsilon\right)\leq \mathbb{P}_{0,s_0}\left(\left\{\frac{\mathcal{N}_n}{\mathcal{D}_n}>\epsilon\right\}\cap E_n\cap F_n\right)+\mathbb{P}_{0,s_0}(E_n^c)+\mathbb{P}_{0,s_0}(F_n^c),
\end{equation}
which we shall use for suitably chosen $E_n$ and $F_n$ having the property $\mathbb{P}_{0,s_0}(E_n^c)\to 0$ and $\mathbb{P}_{0,s_0}(F_n^c)\to 0$.

Denote $S_n(s)=n^{-1}\log R_n(s)$ and note that
$
\mathcal{D}_n=\int_{\mathcal{S}_n} \exp( n S_n(s)) \Pi_n(\dd s).
$
As in \cite{gugu16}, we write
\begin{equation*}
S_n(s)=\frac{1}{2}\frac{1}{n}\sum_{i=1}^{n} \mathcal{W}_i f_s(z_{i})+\frac{1}{2}\frac{1}{n}\sum_{i=1}^{n}\left[  \log\left( 1+f_s(z_{i}) \right) - f_s(z_{i}) \right].
\end{equation*}
Below and in Appendix~\ref{section:proofsremaining} we need the neighbourhoods $V_{s_0,\varepsilon}=\left\{ {s}\in\mathcal{S}: {\| {s}-s_0\|_{\infty}} <\varepsilon \right\}$ for arbitrary $\eps>0$,  
where $\|\cdot\|_{\infty}$ is the usual $L_{\infty}$-norm. We will use these for $\eps=\widetilde{\eps}_n\asymp n^{-\beta}\log^\gamma n$.
We have the following lower bound on $\mathcal{D}_n$,
\[
\mathcal{D}_n\geq  \int_{V_{s_0,\widetilde{\varepsilon}_n}} R_n(s)\Pi_n(\dd s)  \geq \inf_{s\in V_{s_0,\widetilde{\varepsilon}_n}} R_n(s)\times\Pi_n(V_{s_0,\widetilde{\varepsilon}_n}).
\]

Let 
\begin{equation*}
E_n=\left\{\sup_{s\in V_{s_0,\widetilde{\varepsilon}_n}}\left|\frac{1}{n}\sum_{i=1}^{n} \mathcal{W}_i f_s(z_{i})\right|\leq \delta_n\right\},
\end{equation*}
with $\delta_n=\widetilde{\varepsilon}_n^2$. Lemmas \ref{lemma1} and \ref{lemma0} from Appendix~\ref{section:proofsremaining} yield that on $E_n$ one has
\begin{align*}
\inf_{s\in V_{s_0,\widetilde{\varepsilon}_n}}R_n(s)
& \geq \exp\left(- \frac{2\mathcal{K}^2}{\kappa^4} n\widetilde{\varepsilon}_n^2-\frac{n}{2}\sup_{s\in V_{s_0,\widetilde{\varepsilon}_n}}\left|\frac{1}{n}\sum_{i=1}^{n} \mathcal{W}_i f_s(z_{i})\right|\right) \\
& \geq \exp\left(- \frac{2\mathcal{K}^2}{\kappa^4} n\widetilde{\varepsilon}_n^2-\frac{n\widetilde{\varepsilon}_n^2}{2}\right),
\end{align*}
whereas $\mathbb{P}_{0,s_0}(E_n^c)\to 0$. Hence on $E_n$ we have, using the prior mass result of Lemma~\ref{lemma_prior}, 
\[
\mathcal{D}_n\geq  \exp\left(
- \left(\frac{2\mathcal{K}^2}{\kappa^4} +\frac{1}{2}+\overline{C}\right)n\widetilde{\varepsilon}_n^2
\right).
\]
Next we consider $\mathcal{N}_n$, for which we have the trivial upper bound 
\[
\mathcal{N}_n\leq \sup_{{U}_{s_0,\varepsilon_n}^c} R_n(s).
\]
Let
\[
{F_n}=\left\{\sup_{s\in U_{s_0,\varepsilon_n}^c}R_n(s)\leq \exp\left( -c_1 n\varepsilon_n^2 \right)\right\},
\]
with $\eps_n\asymp n^{-\beta}\log^\gamma n$ and the constant $c_1>0$ as in Lemma ~\ref{lemma2}. According to the latter lemma, we have $\mathbb{P}_{0,s_0}(F_n^c)\to 0$.
Putting the above results together, we obtain on $E_n\cap F_n$ the inequality
\[
\frac{\mathcal{N}_n}{\mathcal{D}_n}\leq \exp\left(\left(
\frac{2\mathcal{K}^2}{\kappa^4} +\frac{1}{2}+\overline{C}\right)n\widetilde{\varepsilon}_n^2-c_1n{\varepsilon}_n^2
\right).
\]
Taking $\varepsilon_n = \widetilde{M} \widetilde{\varepsilon}_n,$ with $\widetilde{M}>0$ large enough, gives a positive constant $M$ for which on $E_n\cap F_n$ the inequality
\[
\frac{\mathcal{N}_n}{\mathcal{D}_n}\leq \exp\left(-Mn\widetilde{\varepsilon}_n^2
\right)
\]
holds. It follows that
\[
\mathbb{P}_{0,s_0}\left(\left\{\frac{\mathcal{N}_n}{\mathcal{D}_n}>\epsilon\right\}\cap E_n\cap F_n\right)=0
\]
for all large $n$. We then obtain from \eqref{eq:ab} that
\[
\Pi_n(U_{s_0,\varepsilon_n}^c|\scr{X}_n)\convp 0.
\]
The assertion of the theorem now follows from this and the dominated convergence theorem. The optimal choice for $\beta$ in $\eps_n\asymp n^{-\beta}\log^\gamma n,$ as well as of $\alpha$ in $N_n\asymp n^{\alpha},$ is a consequence of a discussion in Remark \ref{rem:optimal}.
\end{proof}

\subsection{Proofs of the remaining technical results}
\label{section:proofsremaining}

We use the following notation: $M_{\varepsilon}$ will denote the smallest positive integer, such that $2^{M_{\varepsilon}}\varepsilon^2\geq 4 \mathcal{K}^2.$ Note that this definition implies $2^{M_{\varepsilon}}\varepsilon^2\leq 8\mathcal{K}^2,$ so that $M_{\varepsilon}\asymp \log_2(1/{\varepsilon})$  for $\varepsilon\rightarrow 0.$ We also define sets $A_{j,\varepsilon}=\{ s\in\mathcal{S}: 2^j\varepsilon^2\leq \|s-s_0\|_2^2 < 2^{j+1}\varepsilon^2 \}$ and $B_{j,\varepsilon}=\{ s\in\mathcal{S}: \|s-s_0\|_2^2 < 2^{j+1}\varepsilon^2 \}$ for $j=0,1,\ldots,M_{\varepsilon}.$ The measure $Q_n$ will be the uniform discrete probability measure on points $z_i$'s, while $\|\cdot\|_{Q_n}$ will denote the $L_2(Q_n)$-norm. In this appendix we use these sets for $\eps=\widetilde{\eps}_n\asymp n^{-\beta}\log^\gamma n$.

The next lemma verifies the prior mass condition in the proof of our main theorem. This corresponds, roughly speaking, to e.g.\ condition (2.4) in Theorem 2.1 of \cite{ghosal00}. This prior mass condition is crucial in establishing posterior contraction rates and we refer to \cite{ghosal00} for an additional discussion on it. Note that in \cite{gugu16} this condition is verified for another, somewhat artificial, prior.
\begin{lemma}
\label{lemma_prior}
Under the conditions of Theorem \ref{mainthm}, the prior $\Pi_n$ satisfies
\begin{equation}
\label{priorC}
\Pi_n( V_{s_0,\widetilde{\varepsilon}_n} )  \gtrsim e^{- \overline{C} n \widetilde{\varepsilon}_n^2}
\end{equation}
for some constant $\overline{C}>0.$
\end{lemma}

\begin{proof}
Since $\xi_k$'s are independent, we have
\begin{align*}
\Pi_n( V_{s_0,\widetilde{\varepsilon}_n} )&=\Pi_n\left( \bigcap_{k=1}^{N_n} \left\{ \sup_{x\in B_k} |s(x)-s_0(x)|<\widetilde{\varepsilon}_n \right\} \right)\\
&=\prod_{k=1}^{N_n}\Pi_n\left( \sup_{x\in B_k} | s(x) -s_0(x)|<\widetilde{\varepsilon}_n \right) =\prod_{k=1}^{N_n}\Pi_n\left( \sup_{x\in B_k} |\xi_k-s_0(x)|<\widetilde{\varepsilon}_n \right).
\end{align*}
Now, since $s_0$ is H\"older continuous and $B_k$ is of length at most $2m_n/n$, we have
\begin{equation*}
s_0(x)=s_0(a_{k-1})+O\left(\frac{m_n}{n}\right)^\lambda, \quad x\in B_k,
\end{equation*}

where the order term is uniform in $x\in [0,1]$.
\rm
Remember also that by our conditions
\begin{equation}
\label{eq:m}
\left(\frac{m_n}{n}\right)^\lambda \ll \widetilde{\varepsilon}_n.
\end{equation}
Then for all $n$ large enough and all $k,$
\begin{align*}
\Pi_n\left( \sup_{x\in B_k} |\xi_k-s_0(x)|<\widetilde{\varepsilon}_n \right) & \geq \Pi_n\left( |\xi_k-s_0(a_{k-1})|<\frac{\widetilde{\varepsilon}_n}{2} \right)\\
&\geq\textrm{const} \cdot \widetilde{\varepsilon}_n,
\end{align*}
where $\textrm{const}>0$ is 

some constant independent of $k$ and $n$, 
\rm
and the last inequality comes from the fact that $\xi_k$ has a density bounded away from zero on $[\kappa,\mathcal{K}].$ It then follows that
\begin{align*}
\Pi_n( V_{s_0,\widetilde{\varepsilon}_n} )&\geq (\textrm{const} \cdot \widetilde{\varepsilon}_n)^{N_n}\\
&=e^{N_n \log(\textrm{const} \cdot \widetilde{\varepsilon}_n)}.
\end{align*}
We want to show existence of a constant $\overline{C}>0,$ such that for all $n$ large enough,
\begin{equation*}
N_n \log(\textrm{const} \cdot \widetilde{\varepsilon}_n) \geq - \overline{C} n \widetilde{\varepsilon}_n^2.
\end{equation*}
But this is immediate from our conditions on $N_n$ (equivalently, $m_n$) and $\widetilde{\varepsilon}_n,$ e.g.\ for $\overline{C}=1.$ The proof of the lemma is completed.
\end{proof}

The following lemma is an analogue of Lemma~A.2 in \cite{gugu16}. The statement is slightly different, and so is the proof.

\begin{lemma}
\label{lemma1}
Let the conditions of Theorem \ref{mainthm} hold and let $s\in V_{s_0,\widetilde{\varepsilon}_n}.$ Then, uniformly in $s,$
\begin{equation*}
\frac{1}{2}\frac{1}{n}\sum_{i=1}^n \left\{ \log(1+f_s(z_i)) - f_s(z_i) \right\}
\geq - \frac{2\mathcal{K}^2}{\kappa^4} \widetilde{\varepsilon}_n^2.
\end{equation*}
\end{lemma}

\begin{proof}
As in the proof in \cite{gugu16}, one derives
\begin{equation}\label{eq:lbound1}
\frac{1}{2}\frac{1}{n}\sum_{i=1}^n \left\{ \log(1+f_s(z_i)) - f_s(z_i) \right\}
\geq 
-\frac{1}{2n}\sum_{i=1}^n f_s^2(z_i).
\end{equation}
Using the Cauchy-Schwarz inequality, the bounds on $s$ and the fact that $s\in V_{s_0,\widetilde{\varepsilon}_n}$, 
one gets
\begin{align*}
f_s^2(z_i) & = \left( \frac{\int_{z}^{z+1/n}[s_0^2(u)-s^2(u)]\,\dd u}{\int_{z}^{z+1/n}s^2(u)\,\dd u}\right)^2  \leq \frac{\frac{1}{n}\int_{z}^{z+1/n}[s_0^2(u)-s^2(u)]^2\,\dd u}{\kappa^4/n^2} \\
&  \leq  \frac{\frac{4\mathcal{K}^2}{n}\int_{z}^{z+1/n}[s_0(u)-s(u)]^2\,\dd u}{\kappa^4/n^2}  \leq \frac{4\mathcal{K}^2\widetilde{\varepsilon}_n^2}{\kappa^4}.
\end{align*}
The result follows by inserting the latter upper bound into \eqref{eq:lbound1}.
\end{proof}

Lemma~\ref{lemma0} below is a preciser version of Lemma~A.1 in \cite{gugu16}. Its proof is similar in general terms, but differs substantially in the entropy estimates used, as well as some other arguments. In its proof and that of Lemma~\ref{lemma3} we will need the following lemma.
\begin{lemma}\label{lemma:lipschitz}
Let $g:[0,\infty)\to\mathbb{R}$ be H\"{o}lder continuous of order $\lambda>0$ and H\"{o}lder constant $L>0.$ Then 
\[
\left|\int_z^{z+h}g(u)\,\dd u - hg(z)\right|\leq \frac{Lh^{1+\lambda}}{1+\lambda}.
\]
Furthermore, if $z_i=\frac{i}{n}$, for $i=1,\ldots,n$, then 
\[
\left|\int_0^{1}g(u)\,\dd u - \frac{1}{n}\sum_{i=1}^ng(z_i)\right|\leq \frac{L}{(1+\lambda)n^\lambda}.
\]
\end{lemma}

\begin{proof}
For the first  statement we consider
\begin{align*}
\left|\int_z^{z+h}g(u)\,\dd u - hg(z)\right| & = \left|\int_z^{z+h}(g(u)-g(z))\,\dd u\right| \\
& \leq \int_z^{z+h}|g(u)-g(z)|\,\dd u \leq \int_z^{z+h}L(u-z)^{\lambda}\,\dd u  = \frac{Lh^{1+\lambda}}{1+\lambda}.
\end{align*}
For the second one we have, using the first part,
\begin{multline*}
\left|\int_0^{1}g(u)\,\dd u - \frac{1}{n}\sum_{i=1}^ng(z_i)\right| =\left|\sum_{i=1}^n\int_{\frac{i-1}{n}}^{\frac{i}{n}}(g(u) - g(z_i))\,\dd u\right| \\
 \leq \sum_{i=1}^n\left|\int_{\frac{i-1}{n}}^{\frac{i}{n}}(g(u) - g(z_i))\,\dd u\right|  \leq \frac{L}{(1+\lambda)n^\lambda}.
\end{multline*}
The proof is completed.
\end{proof}

\begin{lemma}
\label{lemma0}
Let the conditions of Theorem \ref{mainthm} hold and assume $b_0=0$. Then  
\begin{equation*}
\mathbb{P}_{0,s_0}\left(\sup_{f_s\in\mathcal{F}_{s_0,\widetilde{\varepsilon}_n}}\left| \frac{1}{n} \sum_{i=1}^n \mathcal{W}_i f_s(z_{i}) \right|\geq\delta_n\right) \lesssim \frac{1}{n^{\lambda}{\widetilde{\varepsilon}_n}^2},
\end{equation*}
where $\mathcal{F}_{s_0,\widetilde{\varepsilon}_n}=\{   f_s:\|s-s_0\|_{\infty}<\widetilde{\varepsilon}_n   \}=\{f_s: s \in V_{s_0,\widetilde{\varepsilon}_n}\},$ and $\delta_n$ is an arbitrary sequence of positive numbers, such that $\delta_n \asymp {\widetilde{\varepsilon}_n}^2.$ In particular, as $n\rightarrow\infty,$ the probability on the lefthand side of the above display converges to zero.
\end{lemma}

\begin{proof}
Introduce
\begin{equation*}
g_s(z)=\frac{s_0^2(z)-s^2(z)}{s^2(z)}, \quad \mathcal{G}_{s_0,\widetilde{\varepsilon}_n}=\{g_s: \|s-s_0\|_{\infty}<\widetilde{\varepsilon}_n \}.
\end{equation*}
The function $g_s$ approximates $f_s$  in the following sense:
\begin{equation}
\label{approx*}
|f_s(z_i)-g_s(z_i)|\leq \frac{2\mathcal{K}L}{\kappa^2n^\lambda}, \quad i=1,\ldots,n,
\end{equation}
which can be seen as follows. Every interval $[z_i,z_i+1/n)$ is contained in some bin $B_k,$ so that $s^2$ is constant on $[z_i,z_i+1/n)$. Hence one obtains $f_s(z_i)=n\int_{z_i}^{z_i+1/n}g_s(u)\,\dd u$. It follows from Definition \ref{classX} and boundedness and H\"{o}lder continuity of $s_0$ in Assumption~\ref{standing} that
\[
|g_s(u)-g_s(v)|\leq \frac{2\mathcal{K}L}{\kappa^2}|u-v|^\lambda.
\]
Hence equation~\eqref{approx*} follows from Lemma~\ref{lemma:lipschitz}. Note that the righthand side of \eqref{approx*} is uniform in $s$.

Let
\[
\rho_n=\frac{2\mathcal{K}L}{\kappa^2n^{1+\lambda}}\sum_{i=1}^n|\mathcal{W}_i|.
\]
Then
\begin{equation*}
\sup_{f_s\in\mathcal{F}_{s_0,\widetilde{\varepsilon}_n}}\left| \frac{1}{n} \sum_{i=1}^n \mathcal{W}_i f_s(z_i) \right| \leq \sup_{g_s\in\mathcal{G}_{s_0,\widetilde{\varepsilon}_n}}\left| \frac{1}{n} \sum_{i=1}^n \mathcal{W}_i g_s(z_i) \right| + \rho_n.
\end{equation*}
Hence,
\begin{align*}
\lefteqn{\mathbb{P}_{0,s_0}\left(\sup_{f_s\in\mathcal{F}_{s_0,\widetilde{\varepsilon}_n}}\left| \frac{1}{n} \sum_{i=1}^n \mathcal{W}_i f_s(z_{i}) \right|\geq\delta_n\right)}\\
& \qquad \leq \mathbb{P}_{0,s_0}\left(\sup_{g_s\in\mathcal{G}_{s_0,\widetilde{\varepsilon}_n}}\left| \frac{1}{n} \sum_{i=1}^n \mathcal{W}_i g_s(z_{i}) \right|+\rho_n\geq\delta_n\right) \\
& \qquad \leq \mathbb{P}_{0,s_0}\left(\sup_{g_s\in\mathcal{G}_{s_0,\widetilde{\varepsilon}_n}}\left| \frac{1}{n} \sum_{i=1}^n \mathcal{W}_i g_s(z_{i}) \right|\geq\delta_n/2\right) + \mathbb{P}_{0,s_0}(\rho_n\geq \delta_n/2).
\end{align*}

The Markov inequality gives
\[
\mathbb{P}_{0,s_0}(\rho_n\geq\delta_n/2)\lesssim \frac{1}{n^{\lambda}\delta_n}= \frac{1}{n^{\lambda}{\widetilde{\varepsilon}_n}^2}.
\] 
\rm
Therefore, in order to prove the lemma, it is enough to show that\begin{equation}\label{eq:G}
\mathbb{P}_{0,s_0}\left(\sup_{g_s\in\mathcal{G}_{s_0,\widetilde{\varepsilon}_n}}\left| \frac{1}{n} \sum_{i=1}^n \mathcal{W}_i g_s(z_{i}) \right|\geq\delta_n\right)\lesssim \frac{1}{n^{\lambda}{\widetilde{\varepsilon}_n}^2}.
\end{equation}
We will apply Corollary 8.8 from \cite{geer00} to show that \eqref{eq:G} holds true, since the assertion of that corollary, inequality~(8.30), implies that  
\begin{equation}\label{eq:830}
\mathbb{P}_{0,s_0}\left(\sup_{g_s\in\mathcal{G}_{s_0,\widetilde{\varepsilon}_n}}\left| \frac{1}{n} \sum_{i=1}^n \mathcal{W}_i g_s(z_{i}) \right|\geq\delta_n\right) \leq \tilde c_1\exp(-\tilde c_2\,n\widetilde{\varepsilon}_n^2),
\end{equation}
for some positive constants $\tilde c_1,\tilde c_2$. The righthand side of the above display is asymptotically much smaller than $1/(n^{\lambda}{\widetilde{\varepsilon}_n}^2).$

Application of the mentioned corollary amounts to verification of formulae (8.23)--(8.29) in \cite{geer00}. Exactly as in \cite{gugu16}, conditions (8.23)--(8.27), (8.29) can be satisfied by choosing $R_n=2\mathcal{K}\widetilde{\varepsilon}_n/\kappa^2,$ $K_1=3,$ $\sigma_0^2=2 \ee_{0,s_0}\left[ \mathcal{W}_i^2 e^{|\mathcal{W}_i|/3} \right],$ $K_2=2\mathcal{K}\widetilde{\varepsilon}_n/\kappa^2,$ $C_1=3,$ $K=4K_1 K_2,$ and $C_0=2C,$ with $C$ a universal constant as in Corollary 8.8 in \cite{geer00}. Finally, we need to verify formula
(8.28) in \cite{geer00},
\begin{equation}
\label{8.28}
\sqrt{n}\delta_n \geq C_0 \left( \int_{\delta_n/2^6}^{\sqrt{2}R_n\sigma_0} H_B^{1/2}\left( \frac{u}{\sqrt{2}\sigma_0},\mathcal{G}_{s_0,\widetilde{\varepsilon}_n},Q_n\right) \,\dd u \vee \sqrt{2} R_n \sigma_0 \right).
\end{equation}
Here $H_B\left( \delta,\mathcal{G}_{s_0,\widetilde{\varepsilon}_n},Q_n\right)$ is the $\delta$-entropy with bracketing of $\mathcal{G}_{s_0,\widetilde{\varepsilon}_n}$ for the $L_2(Q_n)$-metric (see Definition 2.2 in \cite{geer00}).

First of all, note that
\begin{equation*}
\sqrt{n}\delta_n \asymp \sqrt{n} \widetilde{\varepsilon}_n^2 \gg R_n \asymp \widetilde{\varepsilon}_n,
\end{equation*}
holds, since we have $n\widetilde{\varepsilon}_n^2\rightarrow\infty.$ Thus it suffices to show
\begin{equation*}
\sqrt{n}\delta_n \geq C_0 \int_{\delta_n/2^6}^{\sqrt{2}R_n\sigma_0} H_B^{1/2}\left( \frac{u}{\sqrt{2}\sigma_0},\mathcal{G}_{s_0,\widetilde{\varepsilon}_n},Q_n\right) \,\dd u.
\end{equation*}
In order to do this, we will upper bound the righthand side by first upper bounding the integrand with a manageable and simple expression, and thereby we obtain a bound on the integral itself. We will use $H_{\infty}$, the entropy for the supremum norm (see Definition 2.3 in \cite{geer00}) and the inequality $H_B(\delta,\mathcal{G},Q)\leq H_\infty(\delta/2,\mathcal{G})$, valid for any collection of functions $\mathcal{G}$ and a probability measure $Q$, see Lemma 2.1 in \cite{geer00}, to obtain
\begin{align*}
\lefteqn{\int_{\delta_n/2^6}^{\sqrt{2}R_n\sigma_0} H_B^{1/2}\left( \frac{u}{\sqrt{2}\sigma_0},\mathcal{G}_{s_0,\widetilde{\varepsilon}_n},Q_n\right) \,\dd u} \\
 & \qquad\leq \int_{\delta_n/2^6}^{\sqrt{2}R_n\sigma_0} H_{\infty}^{1/2}\left( \frac{u}{2\sqrt{2}\sigma_0},\mathcal{G}_{s_0,\widetilde{\varepsilon}_n}\right) \,\dd u\\
&\qquad = 2\sqrt{2}\sigma_0 \int_{0}^{R_n/2-\delta_n/(2^7\sqrt{2}\sigma_0)} H_{\infty}^{1/2}\left( u + \frac{\delta_n}{2^7\sqrt{2}\sigma_0},\mathcal{G}_{s_0,\widetilde{\varepsilon}_n}\right) \,\dd u\\
&\qquad\leq 2\sqrt{2}\sigma_0  \int_{0}^{R_n} H_{\infty}^{1/2}\left( u + \frac{\delta_n}{2^7\sqrt{2}\sigma_0},\mathcal{G}_{s_0,\widetilde{\varepsilon}_n}\right) \,\dd u.
\end{align*}
We will now estimate the entropy in the last integral in the above display. Suppose $u>0$ is fixed. For every $g\in\mathcal{G}_{s_0,\widetilde{\varepsilon}_n}$ construct an approximating function
\begin{equation*}
\widetilde{g}=\sum_{k=1}^{N_n} u \left\lfloor \frac{g(a_{k-1})}{u} \right\rfloor 1_{B_k}.
\end{equation*}
The quality of approximation can be assessed as follows: we have
\begin{align*}
\|g-\widetilde{g}\|_{\infty}&=\max_k \|(g-\widetilde{g})1_{B_k}\|_{\infty}\\
&=\max_k\sup_{x\in B_k} \left|g(x)-g(a_{k-1})+u\frac{g(a_{k-1})}{u}-u \left\lfloor \frac{g(a_{k-1})}{u} \right\rfloor\right|\\
&\leq\max_k\sup_{x\in B_k} |g(x)-g(a_{k-1})|+u.
\end{align*}
Since $s^2$ is constant on $B_k$ and $s_0^2$ is H\"{o}lder, while $B_k$ is of length at most $2m_n/n$, we have
\begin{align*}
\sup_{x\in B_k} |g(x)-g(a_{k-1})|&=\sup_{x\in B_k}\left| \frac{ s_0^2(x) - s_0^2(a_{k-1}) }{s^2(a_{k-1})} \right|\\
&\leq \frac{2\mathcal{K}L}{\kappa^2} \left( \frac{2m_n}{n} \right)^\lambda,
\end{align*}
so that
\begin{equation}
\label{8*}
\|g-\widetilde{g}\|_{\infty} \leq u + \frac{2\mathcal{K}L}{\kappa^2} \left( \frac{2m_n}{n} \right)^\lambda.
\end{equation}
According to our assumptions,
\begin{equation}
\label{bound*}
\left( \frac{m_n}{n} \right)^\lambda \ll \delta_n,
\end{equation}
and hence for any pair of positive constants $A$ and $B$ one eventually has
\begin{equation*}
A\left( \frac{2m_n}{n} \right)^\lambda \leq B\delta_n.
\end{equation*}
This implies that for all $n$ large enough,
\begin{equation*}
H^{1/2}_{\infty}\left( u + \frac{\delta_n}{2^7\sqrt{2}\sigma_0},\mathcal{G}_{s_0,\widetilde{\varepsilon}_n}\right) \leq H^{1/2}_{\infty}\left( u + \frac{2\mathcal{K}L}{\kappa^2} \left( \frac{2m_n}{n} \right)^\lambda  ,\mathcal{G}_{s_0,\widetilde{\varepsilon}_n}\right).
\end{equation*}
The entropy on the righthand side of the above display can be estimated by bounding the corresponding covering number of $\mathcal{G}_{s_0,\widetilde{\varepsilon}_n}$, which can be achieved by counting the number of different $\widetilde{g}$'s. In fact, as we shall see below, although $\mathcal{G}_{s_0,\widetilde{\varepsilon}_n}$ is an infinite set, the number of different $\widetilde{g}$'s is finite and can by easily estimated from above.

Firstly, recall that $\|g\|_{\infty}\leq (2\mathcal{K}/\kappa^2)\widetilde{\varepsilon}_n$ for $g\in \mathcal{G}_{s_0,\widetilde{\varepsilon}_n}$. This implies that there are at most \begin{equation*}
\left\lfloor \frac{4\mathcal{K}}{\kappa^2} \frac{\widetilde{\varepsilon}_n}{u} \right\rfloor+1
\end{equation*}
possible values for $\widetilde{g}(a_0).$ Next, by the triangle inequality and \eqref{8*},
\begin{align*}
|\widetilde{g}(a_k)-\widetilde{g}(a_{k-1})|&\leq |\widetilde{g}(a_k)-g(a_k)|+|g(a_k)-g(a_{k-1})|+|\widetilde{g}(a_{k-1})-g(a_{k-1})|\\
&\leq 2u+  \frac{4\mathcal{K}}{\kappa^2}  L \left( \frac{2m_n}{n} \right)^\lambda+2\|g\|_{\infty}.
\end{align*}
Thus, once $\widetilde{g}(a_{k-1})$ has been determined, $\widetilde{g}(a_k)$ can take at most
\begin{equation*}
\left\lfloor \frac{1}{u} \left(4u+ \frac{8\mathcal{K}}{\kappa^2}  L \left( \frac{2m_n}{n} \right)^\lambda+4\|g\|_{\infty} \right) \right\rfloor+1
\end{equation*}
values. Therefore, in total there can be at most
\begin{equation*}
\left( \left\lfloor \frac{4\mathcal{K}}{\kappa^2} \frac{\widetilde{\varepsilon}_n}{u} \right\rfloor  +1 \right) \left( \left\lfloor  4+ \frac{8\mathcal{K}}{\kappa^2}  \frac{L}{u} \left( \frac{2m_n}{n} \right)^\lambda+\frac{8\mathcal{K}}{\kappa^2} \frac{\widetilde{\varepsilon}_n}{u}  \right\rfloor+1 \right)^{N_n-1}
\end{equation*}
different $\widetilde{g}$'s, which yields an upper bound on the covering number of the set $\mathcal{G}_{s_0,\widetilde{\varepsilon}_n}.$ Using \eqref{bound*}, for large $n$, the logarithm of the above display is bounded by
\begin{equation*}
N_n \log\left( \textrm{const}+ \textrm{const} \frac{\widetilde{\varepsilon}_n}{u} \right)
\end{equation*}
for some constant $\textrm{const}>0$ independent of $n$ and $u.$ This expression gives an upper bound on the entropy of the set $\mathcal{G}_{s_0,\widetilde{\varepsilon}_n}.$ Inserting this upper bound into the entropy integral, we get that for all $n$ large enough,
\begin{align*}
\int_{0}^{R_n} H_{\infty}^{1/2}\left( u + \frac{\delta_n}{2^7\sqrt{2}\sigma_0},\mathcal{G}_{s_0,\widetilde{\varepsilon}_n}\right) \,\dd u & \leq \sqrt{N_n} \int_{0}^{R_n} \log^{1/2}\left( \textrm{const}+ \textrm{const} \frac{\widetilde{\varepsilon}_n}{u} \right) \,\dd u\\
&=\widetilde{\varepsilon}_n \sqrt{N_n} \int_{0}^{2\mathcal{K}/\kappa^2}  \log^{1/2}\left( \textrm{const}+ \textrm{const} \frac{1}{u} \right) \,\dd u \\
&\asymp \widetilde{\varepsilon}_n \sqrt{N_n},
\end{align*}
since the integral in the second line is convergent, because the integrand is bounded by a constant times $u^{-1/2},$ and the latter function is integrable in the neighbourhood of zero.

Summarising the above intermediate calculations, we obtain that in order to prove (8.28) in \cite{geer00}, we need 
\begin{equation}\label{eq:delta}
\sqrt{n}\delta_n \gg \widetilde{\varepsilon}_n \sqrt{N_n}
\end{equation}
and \eqref{bound*} to hold. Both are satisfied with our choice of $\widetilde{\varepsilon}_n$ and $N_n$ (equivalently, $m_n$). Hence all conditions of Corollary~8.8 in \cite{geer00} are satisfied and \eqref{eq:830} follows.
This completes the proof.
\end{proof}
The next lemma, to be used in the proof of Lemma~\ref{lemma2}, is an analogue of Lemma~A.4 in \cite{gugu16}. Its proof is also similar in structure, but differs in details.

\begin{lemma}
\label{lemma3} Let the conditions of Theorem \ref{mainthm} hold and assume $b_0=0$. 

There exist two positive constants $\widetilde{c}_0$ and $\widetilde{C}_0$, depending on $\kappa$, $\mathcal{K}$ and $L$ only, 
\rm
such that for all $n$ large enough and all $s\in A_{j,\varepsilon_n},$ $j=0,1,\ldots,M_{\varepsilon_n},$ we have
\begin{equation*}
\sum_{i=1}^n \exz[Z_{i,n,s}(Y_{i,n})] \leq - \frac{\widetilde{c}_0\kappa^2}{\mathcal{K}^4}2^j\varepsilon_n^2 n +\widetilde{C}_0n^{1-\lambda}.
\end{equation*}
\end{lemma}
\begin{proof}
Below we use the following elementary inequality: for any fixed constant $C>0$ there exists another constant $\widetilde{c}_0>0$, such that for $x\in (-1,C)$, $\log(1+x)-x\leq - \widetilde{c}_0 x^2$ holds.

It follows from Assumption~\ref{standing} that 
\[
\frac{ \int_{[z_i,z_{i+1)}} [s_0^2(u)-s^2(u)]\,\dd u }{ \int_{[z_i,z_{i+1)}} s^2(u)\,\dd u }\leq \frac{\mathcal{K}^2}{\kappa^2}-1=:C>0.
\]
Hence,  for some $\widetilde{c}_0>0$,
\begin{align*}
\exz[Z_{i,n,s}(Y_{i,n})] &=\frac{1}{2}\log\left( 1 + \frac{ \int_{[z_i,z_{i+1)}} [s_0^2(u)-s^2(u)]\,\dd u }{ \int_{[z_i,z_{i+1)}} s^2(u)\,\dd u } \right)\\
&-\frac{1}{2}\frac{ \int_{[z_i,z_{i+1)}} [s_0^2(u)-s^2(u)]\,\dd u }{ \int_{[z_i,z_{i+1)}} s^2(u)\,\dd u } \\
&\leq - \frac{ \widetilde{c}_0  }{2}  \left\{  \frac{ \int_{[z_i,z_{i+1)}} [s_0^2(u)-s^2(u)]\,\dd u }{ \int_{[z_i,z_{i+1)}} s^2(u)\,\dd u } \right\}^2.
\end{align*}
We now focus on the term in braces.
Let $g(u)=\frac{ (s^2(u)-s^2_0(u))^2 }{s^4(u)}$. On bins, and hence on intervals $[z_i,z_{i+1})$, the function $s$ is constant and positive, and by H\"older continuity and boundedness of $s_0$ also $g$ is H\"older continuous. Application of Lemma~\ref{lemma:lipschitz} gives 
\begin{equation}\label{eq:l1}
\int_{[z_i,z_{i+1)}} \frac{ (s^2(u)-s^2_0(u))^2 }{s^4(u)}\,\dd u =\frac{1}{n}\frac{1}{s^4(z_i)}(s_0^2(z_i)-s^2(z_i))^2+O\left(\frac{1}{n^{1+\lambda}}\right).
\end{equation}
Next, by a similar argument,
\begin{equation}\label{eq:l2}
\frac{ \int_{ [z_i,z_{i+1)}} [s^2(u)-s^2_0(u)] \,\dd u}{  \int_{ [z_i,z_{i+1)}} s^2(u)\,\dd u } = \frac{s^2(z_i)-s_0^2(z_i)}{s^2(z_i)}+O\left(\frac{1}{n^\lambda}\right).
\end{equation}
Combination of \eqref{eq:l1} and \eqref{eq:l2} and the bounds in Assumption~\ref{standing} yields
\[
\left\{ \frac{ \int_{ [z_i,z_{i+1)}} [s^2(u)-s^2_0(u)] \,\dd u}{  \int_{ [z_i,z_{i+1)}} s^2(u)\,\dd u }  \right\}^2 - n\int_{[z_i,z_{i+1)}} \frac{ (s^2(u)-s^2_0(u))^2 }{s^4(u)}\,\dd u
=O\left(\frac{1}{n^\lambda}\right).
\]
Consequently, by summation,
\begin{equation}\label{two*}
\frac{1}{n}\sum_{i=1}^n\left\{ \frac{ \int_{ [z_i,z_{i+1)}} [s^2(u)-s^2_0(u)] \,\dd u}{  \int_{ [z_i,z_{i+1)}} s^2(u)\,\dd u }  \right\}^2 - \int_0^1 \frac{ (s^2(u)-s^2_0(u))^2 }{s^4(u)}\,\dd u
=O\left(\frac{1}{n^\lambda}\right).
\end{equation}
Hence, for all $n$ large enough,
\begin{align*}
\sum_{i=1}^n \exz[Z_{i,n,s}(Y_{i,n})] &\leq 
-\frac{ \widetilde{c}_0 n }{2} \int_0^1 \frac{ (s^2(u)-s^2_0(u))^2 }{s^4(u)}\,\dd u+O(n^{1-\lambda})\\
&\leq - \frac{\widetilde{c}_0\kappa^2}{\mathcal{K}^4}2^j\varepsilon_n^2 n+\widetilde{C}_0n^{1-\lambda}.
\end{align*}
Here the first inequality follows from \eqref{two*} and the last inequality from the assumption $s\in A_{j,\varepsilon_n}$. The proof is completed.
\end{proof}

The following lemma is an analogue of Lemma A.3 in \cite{gugu16}. The proof also shares similar general arguments, but differs in details.

\begin{lemma}
\label{lemma2}

Let Assumption~\ref{standing} be satisfied and assume also $b_0=0$. For a fixed and small enough constant $c_1>0$, for any sequence ${\varepsilon}_n\asymp n^{-\beta}\log^\gamma n$ of strictly positive numbers with $\gamma>0$, and any sequence $m_n \asymp n^{1-\alpha}$ (equivalently, $N_n\asymp n^{\alpha}$) as in Theorem \ref{mainthm}, there exist a constant
$c_0>0$ (depending on $\kappa$, $\mathcal{K}$ and $L$) 
and a universal constant $C>0$
for which the inequality 
\rm
\begin{equation}
\label{one*}
\ppz \left( \sup_{s\in U_{s_0,\varepsilon_n}^c} \prod_{i=1}^n \frac{p_{i,n,s}(Y_{i,n})}{ p_{i,n,0}(Y_{i,n}) } \geq \exp\left( -c_1 n\varepsilon_n^2 \right)\right)
\leq C M_{\varepsilon_n} \exp\left( - c_0 n\varepsilon_n^2 \right)
\end{equation}
holds for all $n$ large enough. 
\end{lemma}

\begin{proof}
As in the proof of Lemma A.3 in \cite{gugu16}, one has
\begin{multline}
\label{zero*}
\ppz \left( \sup_{s\in U_{s_0,\varepsilon_n}^c} \prod_{i=1}^n \frac{p_{i,n,s}(Y_{i,n})}{ p_{i,n,0}(Y_{i,n}) } \geq \exp\left( -c_1 n\varepsilon_n^2 \right)\right)\\
=\sum_{j=0}^{M_{\varepsilon_n}} \ppz\left( \sup_{s\in A_{j,\varepsilon_n}} \prod_{i=1}^n \frac{p_{i,n,s}(Y_{i,n})}{ p_{i,n,0}(Y_{i,n}) } \geq \exp\left( -c_1 n\varepsilon_n^2 \right) \right).
\end{multline}
Furthermore, using Lemma~\ref{lemma3} above, by arguments identical to those in the proof of Lemma~A.3 in \cite{gugu16}, we have for all $n$ large enough
\begin{multline}
\label{pbjn}
\ppz\left( \sup_{s\in A_{j,\varepsilon_n}} \prod_{i=1}^n \frac{p_{i,n,s}(Y_{i,n})}{ p_{i,n,0}(Y_{i,n}) } \geq \exp\left( -c_1 n\varepsilon_n^2 \right) \right)\\
\leq\ppz\Biggl( \sup_{s\in B_{j,\varepsilon_n}}  \left| \frac{1}{n} \sum_{i=1}^n \mathcal{W}_i f_s(z_i) \right|
 \geq  \delta_n \Biggr),
\end{multline}
where
\begin{equation}
\label{deltan}
\delta_n=\overline{\delta}2^{j+1}\varepsilon_n^2= \left( \frac{\tilde{c}_0\kappa^2}{\mathcal{K}^4} - \frac{\widetilde{C}_0}{2^j\varepsilon_n^2n^{\lambda}} - \frac{c_1}{2^{j}} \right) 2^{j+1}  \varepsilon_n^2.
\end{equation}
Note that $\overline{\delta}>0$ for all $n$ large enough, by choosing $0<c_1<\tilde{c}_0\kappa^2/(2\mathcal{K}^4)$ and $n^\lambda\varepsilon_n^2\rightarrow\infty$ as $n\rightarrow\infty.$ The former is a restriction on $c_1$ alluded to in the statement of the theorem.

To bound the righthand side in \eqref{pbjn}, we will apply Corollary 8.8 from \cite{geer00}. To that end we need to verify its conditions. Under our assumptions, we have 
\begin{equation*}
\int_0^1\frac{(s_0^2(u)-s^2(u))^2}{s^4(u)}\,\dd u\leq\frac{4\mathcal{K}^2}{\kappa^4}2^{j+1}\varepsilon_n^2,
\end{equation*}
and the $O(n^{-\lambda})$ term in \eqref{two*} is less than $2^{j+1}\varepsilon_n^2$ for all large $n$. Hence, \eqref{two*} yields
\[
\frac{1}{n}\sum_{i=1}^n \left\{ \frac{ \int_{[z_i,z_{i+1)}} [s_0^2(u)-s^2(u)]\,\dd u }{ \int_{[z_i,z_{i+1)}} s^2(u)\,\dd u } \right\}^2 
\leq \left(\frac{4\mathcal{K}^2}{\kappa^4}+1\right) 2^{j+1}\varepsilon_n^2
\]
for all $n$ large enough and all $j=0,1,\ldots,M_{\varepsilon_n}$. Thus, taking
\begin{equation*}
R_n= \left\{\frac{4\mathcal{K}^2}{\kappa^4}+1\right\}^{1/2} 2^{(j+1)/2}\varepsilon_n
\end{equation*}
yields $\sup_{s\in B_{j,\varepsilon_n}} \|f_s\|_{Q_n}\leq R_n.$ This verifies the unnumbered condition in Corollary (8.8) in \cite{geer00}.

With constants $K_1,$ $C,$ $C_0$ and $C_1$ chosen as in the proof of Lemma \ref{lemma0},  $K_2=2\mathcal{K}^2/\kappa^2$ and $K=4K_1K_2,$ it is easy to see that conditions (8.23)--(8.8.27) and (8.29) in \cite{geer00} will be verified.

Finally, we have to check (8.28) in \cite{geer00}, that amounts to checking the three inequalities $\delta_n\leq C_12R_n^2\sigma_0^2/K,$ $\delta_n\leq 8\sqrt{2}R_n\sigma_0$, and
\begin{equation}
\label{8.28bis}
\sqrt{n}\delta_n \geq C_0 \left( \int_{\delta_n/2^6}^{\sqrt{2}R_n\sigma_0} H_B^{1/2}\left( \frac{u}{\sqrt{2}\sigma_0},\mathcal{F}_{s_0,j,\varepsilon_n},Q_n\right) \,\dd u \vee \sqrt{2} R_n \sigma_0 \right),
\end{equation}
where $\mathcal{F}_{s_0,j,\varepsilon_n}=\{f_s:s\in B_{j,\varepsilon_n}\}$, and $\sigma_0^2=2 \ee_{0,s_0}\left[ \mathcal{W}_i^2 e^{|\mathcal{W}_i|/3} \right]$.
 
Both these inequalities follow by taking $\overline{\delta}>0$ in \eqref{deltan} sufficiently small. This can be accomplished by taking $\widetilde{c}_0$ and hence $c_1$ small enough, still obeying the inequality below that equation. Note that for such a small chosen $\widetilde{c}_0$ the upper bound of Lemma~\ref{lemma3}, needed at the beginning of this proof, is still valid. 
\rm

Now we move to verifying \eqref{8.28bis}, the third inequality. This amounts to separately checking that for all $n$ large enough and all $j=0,1,\ldots,M_{\varepsilon_n},$ the inequalities $n\delta_n^2\geq C_0^2 2 R_n^2\sigma_0^2$ and
\begin{equation}
\label{8.28bisbis}
n\delta_n^2\geq C_0^2 \left( \int_{\delta_n/2^6}^{\sqrt{2}R_n\sigma_0} H_B^{1/2}\left( \frac{u}{\sqrt{2}\sigma_0},\mathcal{F}_{s_0,j,\varepsilon_n},Q_n\right) \,\dd u \right)^2
\end{equation}
hold. The first of these two inequalities is again straightforward, because $n\varepsilon_n^2\rightarrow\infty,$ so we move to the second one. As a preliminary step, we will show how the bracketing numbers of the set $\mathcal{F}_{s_0,j,\varepsilon_n}$ for the $L_2(Q_n)$-norm can be bounded by the bracketing numbers of the set $B_{j,\varepsilon_n}.$ Suppose we have a bracket $[\ell,u]$ for functions $s.$ Then it follows directly from the definition that $[f_{u},f_{\ell}]$ is a bracket for functions $f_s.$ Now we will compare the norms, in order to compare sizes of the brackets: let $s_1,s_2$ be two dispersion coefficients. Then, using \eqref{approx*}, the $c_2$-inequality and applying Lemma~\ref{lemma:lipschitz}, we obtain
\begin{align*}
\| f_{s_1} - f_{s_2} \|_{Q_n}^2 & = \frac{1}{n}\sum_{i=1}^n ( f_{s_1}(z_i) - f_{s_2}(z_i) )^2\\
&=\frac{1}{n}\sum_{i=1}^n \left\{ \frac{ s_0^2(z_i) - s_1^2(z_i) + O(n^{-\lambda}) }{s_1^2(z_i)} - \frac{ s_0^2(z_i) - s_2^2(z_i)+O(n^{-\lambda}) }{s_2^2(z_i)} \right\}^2\\
&\leq 2 \frac{1}{n}\sum_{i=1}^n \left\{ \frac{ s_0^2(z_i) - s_1^2(z_i) }{s_1^2(z_i)} - \frac{ s_0^2(z_i) - s_2^2(z_i) }{s_2^2(z_i)} \right\}^2+O\left(n^{-2\lambda}\right)\\
&=2\int_0^1\left\{ \frac{ s_0^2(u) - s_1^2(u) }{s_1^2(u)} - \frac{ s_0^2(u) - s_2^2(u) }{s_2^2(u)} \right\}^2\,\dd u+O\left(n^{-\lambda}\right)\\
&=2\int_0^1 \frac{ s_0^4(u)( s_1^2(u) - s_2^2(u) )^2 }{s_1^4(u)s_2^4(u)}\,\dd u+O\left(n^{-\lambda}\right)\\
&\leq 2\frac{\mathcal{K}^4}{\kappa^8}\|s_1^2-s_2^2\|_2^2+O\left(n^{-\lambda}\right)\\
& \leq 8\frac{\mathcal{K}^6}{\kappa^8}\|s_1-s_2\|_2^2+O\left(n^{-\lambda}\right).
\end{align*}

Taking square roots, we get from the elementary inequality $\sqrt{a+b}\leq\sqrt{a}+\sqrt{b}$ with $a,b\geq 0$ that
\begin{equation*}
\| f_{s_1} - f_{s_2} \|_{Q_n}\leq \sqrt{8} \frac{\mathcal{K}^3}{\kappa^4}\|s_1-s_2\|_2+O\left(n^{-\lambda/2}\right).
\end{equation*}
Suppose $u\geq \delta_n/2^6,$ as in the entropy integral \eqref{8.28bis}. Then for all $n$ large enough, the remainder term in the above display satisfies $O(n^{-\lambda/2})\leq u/2$. For this we need the condition $O(n^{-\lambda/2})\leq \eps_n^2$, i.e.
\begin{equation}\label{eq:extra}
4\beta\leq\lambda
\end{equation}
to be satisfied, which holds under the stipulated assumptions on $\eps_n$.
Furthermore, if also
\begin{equation*}
\|s_1-s_2\|_2 \leq \frac{\kappa^4}{\sqrt{32}\mathcal{K}^3}u,
\end{equation*}
we get by the triangle inequality that $\| f_{s_1} - f_{s_2} \|_{Q_n}\leq u$. It follows that\begin{equation*}
H_B\left( \frac{u}{\sqrt{2}\sigma_0},\mathcal{F}_{s_0,j,\varepsilon_n},Q_n\right) \leq H_B\left( \frac{\kappa^4 u}{8\sqrt{2}\mathcal{K}^3\sigma_0},B_{j,\varepsilon_n},Q_n\right).
\end{equation*}
The entropy on the righthand side, using $H_B\left( \delta,\mathcal{S}_n,Q_n\right) \leq H_{\infty}\left(\delta/2,\mathcal{S}_n\right)$, see Lemma 2.1 in \cite{geer00}, can be further bounded by
\begin{equation}
\label{three*}
H_B\left( \frac{\kappa^2 u}{8\sqrt{2}\mathcal{K}^3\sigma_0},\mathcal{S}_n,Q_n\right) \leq H_{\infty}\left( \frac{\kappa^2 u}{16\sqrt{2}\mathcal{K}^3\sigma_0},\mathcal{S}_n\right).
\end{equation}
\rm
Now recall that for all $s\in\mathcal{S}_n,$ we have $\kappa \leq s \leq \mathcal{K}.$ This and the fact that there are $N_n$ bins in total imply that the minimal number of balls with radii $v$ to cover the set $\mathcal{S}_n$ with respect to the supremum norm is bounded by
\begin{equation*}
\left( 1 + \textrm{const} \frac{1}{v} \right)^{N_n},
\end{equation*}
for some constant $\textrm{const}>0$ independent of $n.$ Hence, the entropy on the righthand side of \eqref{three*} is bounded by
\begin{equation*}
N_n\log\left( 1 + \textrm{const} \frac{1}{u} \right).
\end{equation*}
Now this bound shows that the required condition \eqref{8.28bis} will follow, if we show the inequality
\begin{equation*}
\sqrt{n}\delta_n \gg  \sqrt{N_n} \int_0^{R_n} \log^{1/2}\left( 1 + \textrm{const} \frac{1}{u} \right)\,\dd u.
\end{equation*}

The integral in this display (without loss of generality we take the const equal to $1$) is of order $R_n\log^{1/2}\left(1+{1}/{R_n}\right)$, which can be seen as follows.
\[
\int_0^{R_n} \log^{1/2}\left( 1 + \frac{1}{u} \right)\,\dd u=R_n\sqrt{\log\left(1+\frac{1}{R_n}\right)}+\frac{1}{2}\int_0^{R_n}\frac{1}{\sqrt{\log(1+1/u)}}\frac{1}{u+1}\,\dd u,
\]
where the latter integral is upper bounded by $\log(1+R_n)/\sqrt{\log(1+{1}/{R_n})}$. This upper bound is for small $R_n$ of lower order than $R_n\sqrt{\log(1+{1}/{R_n})}$. Hence it is sufficient to show that 
\[
\sqrt{n}\delta_n \gg  \sqrt{N_n} R_n\log^{1/2}\left(1+\frac{1}{R_n}\right).
\]
However, the latter relationship follows from 
\begin{equation}\label{eq:epsN}
n\varepsilon_n^2 \gg N_n \log(\varepsilon_n^{-1}), 
\end{equation}
which is true under our choice of $N_n$ and $\varepsilon_n$ as given in Theorem~\ref{mainthm}.

\rm
Since we verified all the conditions in Corollary 8.8 in \cite{geer00}, as in the proof of Lemma~A.3 in \cite{gugu16} we can apply it to the righthand side of \eqref{pbjn} to obtain 
\begin{equation*}
\ppz\left( \sup_{s\in A_{j,\varepsilon_n}} \prod_{i=1}^n \frac{p_{i,n,s}(Y_{i,n})}{ p_{i,n,0}(Y_{i,n}) } \geq \exp\left( -c_1 n\varepsilon_n^2 \right) \right) \leq \exp\left( - c_0 n\varepsilon_n^2 \right),
\end{equation*}
where $c_0>0$ can be expressed in terms of previous constants, cf.\ \cite{geer00} and \cite{gugu16}.
This inequality and \eqref{zero*} yield the statement of the lemma.
\end{proof}

\begin{rem}\label{rem:optimal}
The optimal choice for $\beta$ in $\eps\asymp n^{-\beta}\log^\gamma n$ in the statement of the lemma results from coupling the conditions \eqref{eq:extra} and \eqref{eq:epsN} with
our earlier restrictions {\color{red} \eqref{eq:m}}, \eqref{bound*} and \eqref{eq:delta} on $m_n$, $N_n$, $\varepsilon_n$ and $\widetilde{\varepsilon}_n$. 

Indeed, putting  $N_n\asymp n^\alpha$ gives from the coupling of these five conditions, 
in the above presented order, 
$4\beta\leq \lambda$, $1-2\beta\geq \alpha$ and $\gamma>1$, $\beta\leq\alpha\lambda$, $\beta\leq \alpha\lambda/2$ and $\beta\leq (1-\alpha)/2$. The third and the last conditions can be omitted, which leads to the conditions $4\beta\leq \lambda$, $\beta\leq (1-\alpha)/2$ and  $\beta\leq \alpha\lambda/2$ on $\alpha$ and $\beta$.

It turns out that the maximal possible $\beta$ is $\beta=\frac{\lambda}{4}$, this attains the first condition and also satisfies the other two, for $\alpha \in [\frac{1}{2},1-\frac{\lambda}{2}]$.
\end{rem}

\section*{Acknowledgments}
The research leading to the results in this paper has received funding from the European Research Council under ERC Grant Agreement 320637. The authors would like to thank the referee for constructive comments.

\bibliographystyle{plainnat}

\end{document}